\documentclass[journal]{IEEEtran}
\usepackage{cite}
\usepackage{amsthm}
\usepackage{amsmath,amssymb,amsfonts}
\usepackage{graphicx}
\usepackage{algorithm,algorithmic}
\usepackage{hyperref}
\usepackage{textcomp}
\usepackage{placeins}
\newtheorem{theorem}{Theorem}[section]
\newtheorem{assumption}[theorem]{Assumption}

\newtheorem{definition}[theorem]{Definition}
\newtheorem{lemma}[theorem]{Lemma}
\newtheorem{corollary}[theorem]{Corollary}
\newtheorem{proposition}[theorem]{Proposition}

\newtheorem{remark}[theorem]{Remark}

\def\BibTeX{{\rm B\kern-.05em{\sc i\kern-.025em b}\kern-.08em
    T\kern-.1667em\lower.7ex\hbox{E}\kern-.125emX}}

\begin{document}
\title{Localized Stabilization of Transport PDEs by Interior Flux Feedback}
\author{Constantinos Kitsos and Ian R. Manchester, 
\thanks{The authors are with the Australian Centre for Robotics and School of Aerospace, Mechanical, and Mechatronic Engineering, University of Sydney, Sydney NSW 2006, Australia (emails: \{konstantinos.kitsos, ian.manchester\}@sydney.edu.au) }}

\maketitle
\begin{abstract}
We study stabilization of multidimensional continuity equations with source terms on bounded domains by means of localized interior flux feedback. The feedback is prescribed through the divergence of the flux and is chosen so that the error with respect to a reference profile satisfies a transport equation with localized damping. The main geometric condition is a finite-time characteristic damping inequality, requiring relevant characteristics to accumulate a uniform amount of damping over a time horizon. This condition is shown to yield exponential stability in \(L^2\) of the error, under a gain condition relating localized damping to compressive amplification of the transport field. Lyapunov-type entrance conditions ensure characteristic damping on support-restricted families of trajectories. A weighted Lyapunov functional provides a differential Lyapunov criterion and an input-to-state (ISS) estimate with respect to additive perturbations. We also discuss elliptic right-inverse realizations of the feedback flux and extend the characteristic damping argument to velocity fields depending nonlinearly on the state. A two-dimensional example finally illustrates the geometric, gain, and realization conditions. 
\end{abstract}


\section{Introduction}
Localized stabilization, understood here as stabilization by feedback acting only on a prescribed subregion of the spatial domain, is a central theme in the control of distributed parameter systems. For hyperbolic partial differential equations (PDEs), such problems are intrinsically geometric as the decay or observability depends on whether propagating rays or characteristics encounter the controlled (or observed) region within a suitable time horizon. This principle underlies the classical observability and stabilization theory of wave equations, see the works of \cite{Lions1988}, the geometric control condition of \cite{Bardos1992}, and decay estimates for locally damped waves \cite{HarauxZuazua1988,Zuazua2005}.

The present paper develops a related geometric stabilization approach for continuity equations, viewed here as first-order balance laws for a transported density in the sense of the classical balance-law framework \cite{Dafermos2016}. Transport and continuity equations model the evolution of distributed densities in mass redistribution, traffic flow, crowd motion, multi-agent systems, and robotic logistics, see \cite{Maury2010,PiccoliRossi2013,Cortes2004}. In dynamic optimal transport, the continuity equation appears as a constraint in a finite-horizon optimization problem over density--flux pairs, as in the Benamou--Brenier formulation \cite{BenamouBrenier2000}, see also \cite{Villani2003,Santambrogio2015}. Recent works have also studied steering or optimizing density trajectories under input and density constraints in formulations related to optimal transport
\cite{ChenGeorgiouPavon2016,WuRantzer2024}.

A related direction has concerned stabilization and control of transport equations in the context of control of PDEs. Stability estimates for continuity equations were developed in \cite{KarafyllisKrstic2020} (one-dimensional case, \(L^p\) framework) and in \cite{ColomboMercierRosini2009} (scalar balance laws in several space dimensions, \(BV\) framework), while controllability and feedback problems for conservation laws and hyperbolic systems were considered, for example,
in \cite{BastinCoron2016,PrieurWinkinBastin2008,KarafyllisPapageorgiou2019}, see also \cite{CoronWang2012} for controllability of scalar conservation laws via velocity perturbations. Backstepping methods constitute a major approach to boundary control of transport and hyperbolic PDEs \cite{KrsticSmyshlyaev2008}, while Lyapunov-based distributed and sampled-data \(H_\infty\) control methods for transport-reaction systems were studied in \cite{FridmanBarAm2013}. Localized controllability through velocity fields of the continuity equation was studied in \cite{DuprezMoranceyRossi2019}, and localized control of nonlocal continuity equations has also been considered in \cite{PogodaevRossi2026}. The present work differs from these directions in both actuation mechanism and stability argument. In fact, we consider a multidimensional continuity equation, in which the feedback is imposed through the divergence of an interior flux, and is designed to generate localized damping in the error dynamics.

The error with respect to a reference (target) profile satisfies an equation with localized damping. The control question is how such a damping term, supported only on part of the domain, can stabilize the transported error. We address this by introducing a finite-time characteristic damping condition, which measures damping accumulated by each relevant characteristic over a finite horizon. This condition yields a finite-time contraction estimate for the closed-loop error and is imposed only on the characteristics carrying the initial error and the perturbations. The feedback is specified through the divergence of the control flux, while the realization of the feedback flux is treated separately via elliptic right-inverse constructions. This approach separates the characteristic damping argument from actuator realization.

The contribution of the paper is summarized as follows. First, we formulate a finite-time characteristic damping condition for localized stabilization of the multidimensional continuity equation and prove exponential stability in \(L^2\) under an explicit gain condition balancing localized damping against compressive amplification. Second, we provide constructive Lyapunov-type entrance conditions ensuring the characteristic damping property on support-restricted families of trajectories. These conditions are directly checkable from the velocity field, the localization function, and the support family. We also introduce a weighted Lyapunov functional, with a weight defined through a Lyapunov-type function whose sublevel sets determine the active region, and derive ISS estimates for additive perturbations. Third, we separate the characteristic damping argument from the realization of the feedback flux and discuss elliptic constructions. Finally, we extend the arguments based on characteristic damping to density-dependent velocity fields and obtain nonlinear stabilization criteria under uniform finite-time entrance conditions.

The paper is organized as follows. Section \ref{sec:model} introduces the problem setting. Section \ref{sec:geometry} develops the characteristic geometry and the geometric conditions. Section \ref{sec:stabilization} proves exponential stabilization and derives the weighted Lyapunov and ISS estimates. Section \ref{sec:flux} discusses the realization of the localized feedback flux. Section \ref{sec:nonlinear} extends the approach to density-dependent velocity fields and Section \ref{sec:example} presents an illustrative example.

\paragraph*{Notation}
We write \(\mathbb R_{\ge0}:=[0,+\infty)\). The Euclidean norm and inner product in \(\mathbb R^m\) are denoted by \(|\cdot|\) and \(x\cdot y\), respectively. We denote \(f_+:=\max\{f,0\},\) \(f_-:=\max\{-f,0\}\) for scalar functions \(f\) (hence, \(f=f_+-f_-\)). We use the standard Lebesgue, Sobolev, and trace spaces
\(L^p(\Omega)\), \(H^1(\Omega)\), \(H^1_0(\Omega)\), and
\(H^{\pm1/2}(\partial\Omega)\), all endowed with their usual norms. We denote by \(C_c^\infty(A)\) the space of smooth functions compactly supported in \(A\) and for vector fields, \(W^{1,\infty}(\Omega;\mathbb R^m)\) denotes the space of essentially bounded vector fields with essentially bounded weak derivatives. For \(f\in L^p(\Omega)\), \(\operatorname{ess\,supp}f\) denotes the essential support of \(f\), namely, the complement in \(\Omega\) of the largest open set on which \(f=0\) a.e. We write \(A\Subset B\) when  \(\overline A\) is compactly contained in \(B\).

\section{Problem Setting}
\label{sec:model}
In this section, we present the model, the main hypotheses, and the stabilization problem we aim to solve.

Let \(\Omega\subset\mathbb R^m\); \(m\ge 1\), be a bounded domain with \(C^2\) boundary \(\partial\Omega\). We consider a mobile density
\[
\rho=\rho(t,x),\qquad (t,x)\in\mathbb R_{\ge 0}\times\Omega,
\]
governed by the continuity equation
\begin{align}
\partial_t\rho+\nabla\cdot(\rho u+J_d)=w \qquad \text{in } \mathbb R_{\ge 0}\times\Omega,
\label{eq:rho_balance}
\end{align}
 where \(u=u(t,x)\) is a prescribed velocity field, \(J_d=J_d(t,x)\) is an interior control flux, and \(w=w(t,x)\) is a distributed perturbation.

The objective is to stabilize \(\rho\) towards a prescribed reference profile \(\rho^\star=\rho^\star(x),\) supported for example in a target region \(\Sigma_d\Subset\Omega\) (drop region in logistics) by appropriate design of the interior control flux \(J_d\).

The error 
\begin{align}
e:=\rho-\rho^\star.
\label{eq:error_def}
\end{align} 
satisfies a transport equation with localized damping
\begin{align}
\partial_t e+\nabla\cdot(ue)+\nabla\cdot J_d = w-r^\star,
\label{eq:error_before_feedback}
\end{align}
where \begin{align}
r^\star:=\nabla\cdot(\rho^\star u)\label{eq:rho_star}\end{align} is the transport residual of the reference profile.
We impose the feedback law through the divergence constraint
\begin{align}
\nabla\cdot J_d= \kappa\sigma e-r^\star,
\label{eq:div_law}
\end{align}
 where \(\kappa>0\) is the feedback gain and \(\sigma=\sigma(t,x)\) is a nonnegative localization function determining where the feedback acts and it is positive on an active region \(V_d(t) \subset \Omega\), to be specified later. This formulation may be understood as feedback through an actuator flux, where \(\nabla\cdot J_d\) is the quantity entering the balance law, while the realization of a corresponding flux \(J_d\) is addressed in Section \ref{sec:flux}.

Substituting \eqref{eq:div_law} into \eqref{eq:error_before_feedback}, the closed-loop error equation becomes
\begin{align}
\partial_t e+\nabla\cdot(ue) = -\kappa\sigma e+w \qquad \text{in } \mathbb R_{\ge 0}\times\Omega.
\label{eq:error_closed_loop}
\end{align}

The control divergence law \eqref{eq:div_law} is chosen so that the error equation contains \(-\kappa\sigma e\), which stands for the damping term. The term \(-r^\star\) is a feedforward correction compensating for the fact that the prescribed reference profile need not be transported by \(u\) but is allowed to be a smooth compactly supported profile inside the target region \(\Sigma_d\). Note here that the perturbation \(w\) may also absorb errors in the realization of the flux divergence.

\begin{remark} \label{rem:mass_conservation}
The feedback law we chose above is not mass preserving in general. Indeed, when \(w=0\) and \(\rho^\star\) is compactly supported in the interior, integrating the closed-loop error equation gives \(\frac{d}{dt}\int_\Omega \rho(t,x)dx=-\kappa\int_\Omega \sigma(x)e(t,x) dx.\) Thus, the actuator can be understood as being able to remove or inject density through the prescribed divergence, or possibly as exchanging mass with an external compartment not represented in the PDE. This is appropriate for problems where the objective is convergence to a reference profile beyond typical redistribution of a conserved total mass. In robotic logistics, \(\rho\) may represent agents or material and the feedback divergence can be interpreted as localized activation/deactivation, loading/unloading, or exchange with a depot/buffer. 
If we require strict conservation of mass, we may alternatively consider a model with coupled transport equations involving transported material, unloaded agents, and loaded agents. In such systems, conservation can be imposed through the reaction couplings corresponding to pickup and drop-off mechanisms that preserve the total amount of material and the total number of agents. Extending the present framework to such coupled transport systems, inspired by Lyapunov methods as in \cite{KitsosBesanconPrieur2022}, is left for future work.
\end{remark}

\begin{remark}
Given a reference profile \(\rho^\star\in C_c^\infty(\Sigma_d)\) and velocity field \(u(t,\cdot)\in W^{1,\infty}(\Omega;\mathbb R^m)\), the feedforward residual satisfies \(\rho^\star u(t,\cdot)\in W^{1,\infty}(\Omega;\mathbb R^m)\) and \(\operatorname{supp}(\rho^\star u(t,\cdot))\subset \operatorname{supp}\rho^\star\Subset\Sigma_d.\) Therefore, \(\operatorname{supp} r^\star(t,\cdot)\subset\operatorname{supp}\rho^\star\Subset\Sigma_d\)
in the sense of distributions. Thus, if \(\Sigma_d\subset V_d(t)\) for all \(t\ge0\), the feedforward residual is localized in the active region.
\end{remark}

\begin{remark}
The stabilization analysis is carried out at the level of the  error equation \eqref{eq:error_closed_loop} and the particular vector field \(J_d\) constructing \eqref{eq:div_law} is not unique. The solvability and interpretation of the flux realization are discussed separately in Section~\ref{sec:flux}.
\end{remark}

We assume the following (minimal) regularity on the reference profile and the prescribed velocity field.

\begin{assumption}
\label{ass:domain_u}
The domain \(\Omega\subset\mathbb R^m\) is bounded with \(C^2\) boundary. The velocity field satisfies
\[
u\in L^\infty_{\mathrm{loc}}\bigl(\mathbb R_{\ge 0};W^{1,\infty}(\Omega;\mathbb R^m)\bigr), \qquad
\nabla\cdot u\in L^\infty(\mathbb R_{\ge 0}\times\Omega),
\]
and the boundary condition (impermeability)
\begin{align}
u(t,\cdot)\cdot n=0
\qquad
\text{on }\partial\Omega
\label{eq:impermeability}
\end{align}
holds in the trace sense for a.e. \(t\ge0\).

The reference profile satisfies
\[
\rho^\star\in L^2(\Omega),\quad
r^\star(t,\cdot):=\nabla\cdot(\rho^\star u(t,\cdot))
\in L^2_{\mathrm{loc}}(\mathbb R_{\ge 0};L^2(\Omega)).
\]
The initial error and perturbation satisfy
\[
e_0:=\rho_0-\rho^\star\in L^2(\Omega),\qquad
w\in L^2_{\mathrm{loc}}(\mathbb R_{\ge 0};L^2(\Omega)).
\]
Finally,
\[
\sigma\in L^\infty(\mathbb R_{\ge 0}\times\Omega),\quad
\sigma(t,x)\ge0,
\quad\text{for a.e. }(t,x)\in\mathbb R_{\ge 0}\times\Omega.
\]
\end{assumption}

\begin{remark}
A simple sufficient condition to meet the assumed regularity of \(r^\star\) is
\(\rho^\star\in H^1(\Omega),\) \(u\in L^\infty_{\mathrm{loc}}\bigl(\mathbb R_{\ge 0};W^{1,\infty}(\Omega;\mathbb R^m)\bigr).\) In the example later, \(\rho^\star\in C_c^\infty(\Sigma_d)\), so the condition is automatic under the assumed regularity of \(u\). Condition \eqref{eq:impermeability} is a boundary condition ensuring that the characteristic flow remains inside \(\Omega\), so no boundary condition is required for the error equation \eqref{eq:error_closed_loop}.
\end{remark}

Let \(X(s;t,x)\) denote the characteristic flow associated with the velocity field \(u\), namely,
\begin{align}
\frac{d}{ds}X(s;t,x)=u(s,X(s;t,x)),
\qquad
X(t;t,x)=x.
\label{eq:char_flow}
\end{align}

Under Assumption \ref{ass:domain_u}, the flow is understood in the Carath\'eodory sense and the Lipschitz regularity in \(x\) yields a unique Carath\'eodory flow. The flow extends to a Lipschitz homeomorphism of \(\overline\Omega\) onto itself, with Lipschitz inverse \(X(t;s,\cdot)\) (its restriction maps \(\Omega\) onto \(\Omega\)). It is differentiable a.e. in \(x\), and the Jacobian determinant 
\[
J(s;t,x):=\det D_x X(s;t,x)
\]
of the flow is defined for a.e. \(x\in\Omega\) and satisfies Liouville's formula
\begin{align}
J(s;t,x) = \exp\left( \int_t^s\nabla\cdot u(\tau,X(\tau;t,x)) d\tau \right)
\label{eq:liouville_formula}
\end{align}
for a.e. \(x\in\Omega\). Consequently, for every \(\psi\in L^1(\Omega)\),
\(\int_\Omega \psi(y) dy=\int_\Omega \psi(X(s;t,x))J(s;t,x)dx.\)

For smooth data, let \(0\le s_0\le s_1\) and consider the characteristic
\(s\mapsto X(s;s_0,x)\). Along this characteristic, the closed-loop error
equation \eqref{eq:error_closed_loop} gives
\begin{align}
&\frac{d}{ds}e(s,X(s;s_0,x))
\notag\\&\quad =
-\bigl(\nabla\cdot u+\kappa\sigma\bigr)(s,X(s;s_0,x)) e(s,X(s;s_0,x)) \notag\\
&\quad+ w(s,X(s;s_0,x)).
\label{eq:char_error_ode}
\end{align}
It follows that
\begin{align}
&e(s_1,X(s_1;s_0,x)) \notag\\
&= E(s_1,s_0;x)e(s_0,x) \notag\\
&\quad+ \int_{s_0}^{s_1} e^{-\int_{s}^{s_1} (\nabla\cdot u+\kappa\sigma)(\tau,X(\tau;s_0,x)) d\tau}
w(s,X(s;s_0,x)) d s,
\label{eq:char_error_formula}
\end{align}
where
\begin{align}
E(s_1,s_0;x):= e^{-\int_{s_0}^{s_1}
(\nabla\cdot u+\kappa\sigma) (\tau,X(\tau;s_0,x)) d\tau}.
\label{eq:E_defintion}
\end{align}

We next discuss the well-posedness property of the closed-loop error equation \eqref{eq:error_closed_loop}, which is standard for linear transport equations with sufficiently regular velocity fields and follows from the method of characteristics, see, for example, \cite[Ch. 3]{Evans2010} or \cite[Ch. 8]{LeVeque2002}.

\begin{proposition}
\label{prop:wellposedness}
Suppose Assumption \ref{ass:domain_u} holds. Then, for every \(T>0\),
\(e_0\in L^2(\Omega)\), \(
w\in L^2(0,T;L^2(\Omega)),\)
the closed-loop equation \eqref{eq:error_closed_loop} with initial data \(e_0(x):=e(0,x)\), admits a unique mild solution \(e\in C([0,T];L^2(\Omega)).\)
Moreover, the solution depends continuously on \((e_0,w)\).
\end{proposition}

\begin{proof}
For smooth data, the characteristic formula \eqref{eq:char_error_formula} gives the solution. Since, by Assumption \ref{ass:domain_u}, \(u\in L^\infty(0,T;W^{1,\infty}(\Omega;\mathbb R^m))\) and \(\nabla\cdot u\in L^\infty((0,T)\times\Omega)\), the associated flow is Lipschitz with Lipschitz inverse in the spatial variable for each pair of times. Hence, the homogeneous evolution family is bounded on \(L^2(\Omega)\) on \([0,T]\). To see this, for \(0\le s\le t\le T\), define \(U(t,s)\phi\) by
\[
(U(t,s)\phi)(X(t;s,x))=\phi(x)E(t,s;x),
\]
where \(E(t,s;x)\) is defined by \eqref{eq:E_defintion} with \(s_0=s\) and \(s_1=t\). By Liouville's formula,
\[\begin{aligned}
&\|U(t,s)\phi\|_{L^2(\Omega)}^2\\&=\int_\Omega |\phi(x)|^2 e^{-\int_s^t\nabla\cdot u(\tau,X(\tau;s,x)) d\tau-2\kappa\int_s^t\sigma(\tau,X(\tau;s,x))d\tau}dx \\  &\le e^{ \int_s^t \|(\nabla\cdot u(\tau,\cdot))_-\|_{L^\infty(\Omega)} d\tau}\|\phi\|_{L^2(\Omega)}^2.
\end{aligned}
\]
Thus, \(U(t,s)\) is bounded on \(L^2(\Omega)\) on every finite time interval and Duhamel's formula \(e(t)=U(t,0)e_0+\int_0^tU(t,s)w(s) ds\) then gives a bounded input-to-state map from \(L^2(0,T;L^2(\Omega))\) into \(C([0,T];L^2(\Omega))\). Hence, 
\(e\in C([0,T];L^2(\Omega))\) and uniqueness and continuous dependence follow from the estimate. For general data in \(L^2(\Omega)\), we invoke density arguments.
\end{proof}

We summarize below the main problem we wish to solve.
\paragraph*{Control objective}
Our objective is to design the localized divergence feedback \eqref{eq:div_law} and to provide geometric conditions for velocity field \(u\), localization function \(\sigma\), and active region \(V_d\), along with feedback gains \(\kappa\), such that the closed-loop error equation \eqref{eq:error_closed_loop} satisfies the ISS estimate
\begin{align*}
\|e(t)\|_{L^2(\Omega)}^2 &\le C e^{-\alpha t}\|e_0\|_{L^2(\Omega)}^2 \notag\\ &\quad+
C\int_0^t e^{-\alpha(t-s)}\|w(s)\|_{L^2(\Omega)}^2\,ds,
\end{align*}
for some \(C,\alpha>0\) (implying exponential stabilization for the unperturbed case \(w \equiv 0\)). The stability problem is treated separately from the realization of the flux \(J_d\), and the same geometric condition is extended to density-dependent velocity fields \(u[\rho]\).

\begin{remark}
A central difficulty addressed here, complementarily to \cite{DuprezMoranceyRossi2019,PogodaevRossi2026,
KarafyllisKrstic2020}, is the following: given a damping term \(-\kappa\sigma e\) supported only on \(V_d\Subset\Omega\), under what geometric conditions on the characteristics, and under what explicit gain condition on \(\kappa\), does exponential decay in \(L^2\) hold uniformly for all initial errors supported on the relevant support family? Answering this, as in the sequel, involves quantifying the damping accumulated by each relevant characteristic over a finite horizon and balancing it against compressive amplification of the transport field. This constitutes a geometric difficulty that is the core of the present analysis.
\end{remark}

\section{Characteristic Geometry and Finite-Time Damping}\label{sec:geometry}

In this section, we introduce the geometric conditions that render localized feedback effective. Since the feedback acts through the localization function \(\sigma\), stabilization requires the relevant characteristics to encounter the region where \(\sigma\) is positive. This requirement is first expressed through a finite-time characteristic damping condition and then verified by a Lyapunov-type inequality.

\subsection{Finite-Time Damping Along Characteristics}

For \(T>0\), let us define the accumulated damping along the characteristic starting
from \(x\) at time \(t\) by
\begin{align}
\mathcal D_T(t,x) := \int_t^{t+T}\sigma(s,X(s;t,x)) ds,
\label{eq:char_damping_int}
\end{align}
for \(t\ge0,\ x\in\Omega.\)

\begin{definition}
\label{def:characteristic_damp}
Let \(\mathcal A(t)\subset\Omega\), \(t\ge0\), be a family of measurable sets.
We say that the finite-time characteristic damping condition holds on
\(\mathcal A(t)\) over the horizon \(T>0\) if there exists \(m_T>0\), such that
\begin{align}
\mathcal D_T(t,x)\ge m_T, \qquad t\ge0,\ x\in\mathcal A(t).
\label{eq:finite_time_char_damping}
\end{align}
When \(\mathcal A(t)=\Omega\), we say that the global finite-time characteristic damping condition holds.
\end{definition}

Condition \eqref{eq:finite_time_char_damping} requires every relevant characteristic to accumulate a uniformly positive amount of damping over \([t,t+T]\) and guarantees sufficient total damping without requiring the characteristic to remain in a fixed active region. The constant \(m_T\) depends on the velocity field \(u\), the localization function \(\sigma\), the horizon \(T\), and the family of characteristics under consideration. In Subsection \ref{subsec:Active_region} below, we provide a constructive method to check \eqref{eq:finite_time_char_damping}.

\subsection{Forward-Invariant Support Families} \label{subsec:forw_invariant}

A global damping condition imposed on every characteristic in \(\Omega\), i.e., \(\mathcal A(t)=\Omega \) in \eqref{eq:finite_time_char_damping}, may be stronger than necessary when the error is supported only in part of the domain. We, therefore, introduce support-restricted families generated by the characteristic flow. 

Let \(K_0\subset\Omega\) be a measurable set containing the essential support of the initial error \(e_0\). In the unperturbed case \(w=0\), the support transported from the initial error is contained in
\[\mathcal K_0(t):=X(t;0,K_0).
\]
By the flow property,
\[X(t;s,\mathcal K_0(s))=\mathcal K_0(t), \qquad t\ge s\ge0.
\]
In the perturbed case, if the perturbation \(w\) is supported in a measurable family \(W(s)\subset\Omega\), namely,
\[\operatorname{ess\,supp}w(s,\cdot)\subset W(s),\qquad\text{for a.e. }s\ge0,
\]
the relevant support at time \(t\) is then contained in the forward reachability set
\[X(t;0,K_0)\cup \bigcup_{0\le s\le t}X(t;s,W(s)).\] 
Accordingly, it is convenient to work with any measurable family \(\mathcal K(t)\subset\Omega\) satisfying the forward-invariance condition
\begin{align}
X(t;s,\mathcal K(s))\subset\mathcal K(t),\qquad 0\le s\le t,
\label{eq:support_family_forward}
\end{align}
and containing the support of the initial error and the forward images of the  supports \(W(s)\) of the perturbation.
\begin{lemma}
\label{lem:support_propagation}
Let \(\mathcal K(t)\subset\Omega\), \(t\ge0\), be a family of measurable sets satisfying \eqref{eq:support_family_forward}
and let \(e\) be the mild solution of \eqref{eq:error_closed_loop}. Assume that
\begin{align}
\operatorname{ess\,supp}e_0\subset\mathcal K(0), 
\label{eq:support_e0}
\end{align}
and
\begin{align}
\operatorname{ess\,supp}w(t,\cdot)\subset\mathcal K(t),\qquad\text{for a.e. }t\ge0.
\label{eq:support_w}
\end{align}
Then,
\begin{align}
\operatorname{ess\,supp}e(t,\cdot)\subset\mathcal K(t), \qquad\text{for a.e. }t\ge0. \label{eq:support_e}
\end{align}
\end{lemma}

\begin{proof}
For smooth data, the backward characteristic representation as in \eqref{eq:char_error_formula} gives
\begin{align*}
e(t,y) &= E(t,0;X(0;t,y))e_0(X(0;t,y))\\
&\quad+ \int_0^t E(t,s;X(s;t,y))w(s,X(s;t,y)) ds,
\end{align*}
where \(E(t,s;z)\) is defined in \eqref{eq:E_defintion}. The initial contribution can be nonzero at \(y\) only if \(X(0;t,y)\in \operatorname{ess\,supp}e_0,\)
or equivalently \(y\in X(t;0,\operatorname{ess\,supp}e_0) \subset X(t;0,\mathcal K(0))
\subset \mathcal K(t).\)
Similarly, the contribution generated by \(w(s,\cdot)\) at time \(s\) can be nonzero at \(y\) only if \(X(s;t,y)\in \operatorname{ess\,supp}w(s,\cdot),\)
meaning \(y\in X(t;s,\operatorname{ess\,supp}w(s,\cdot)) \subset X(t;s,\mathcal K(s)) \subset\mathcal K(t)\)
for a.e. \(s\in[0,t]\). Hence, \(e(t,\cdot)\) is supported in \(\mathcal K(t)\), up to null sets.

For general data of class \(L^2\), the conclusion follows by the same characteristic formula for the associated evolution family, or alternatively, by approximation. The Lipschitz flow and its Lipschitz inverse map null sets to null sets, so the inclusions are preserved in the essential-support sense.
\end{proof}

\begin{remark}
In the case where the error is confined to \(\mathcal K(t)\), condition \eqref{eq:support_w} implies that the perturbation is injected only along this family of characteristics. If \(w\) acts outside \(\mathcal K(t)\), then we must either impose a global characteristic damping condition or enlarge \(\mathcal K(t)\) to contain the support of the perturbation and its forward image. Notice also that the family \(\mathcal K(t)\) is fixed independently of the feedback gain and is not part of the feedback law.
\end{remark}

\subsection{Active Region and Lyapunov-Type Entrance}\label{subsec:Active_region}

We introduce here the active region where the localized feedback is effective. We introduce a Lyapunov-type function
\[
\Phi\in C^1(\mathbb R_{\ge 0}\times\overline\Omega)
\] 
and let \(a_d\in\mathbb R\) be such that the active region defined by the sublevel set
\begin{align}
V_d(t):=\{x\in\Omega:\Phi(t,x)\le a_d\}
\label{eq:active_region_closed}
\end{align}
is nonempty for all \(t\ge0\).

The transport derivative of \(\Phi\) along \(u\) is denoted by
\begin{align}
D_u\Phi(t,x) := \partial_t\Phi(t,x)+u(t,x)\cdot\nabla\Phi(t,x).
\label{eq:DuPhi}
\end{align}

Note here that for the geometric results, we work with representatives of \(u,\sigma\), such that the stated inequalities hold pointwise on the relevant sets. The same conclusions hold under the corresponding a.e. assumptions after modifying representatives on null sets.
\begin{assumption}
\label{ass:support_lyapunov}
Let \(\mathcal K(t)\subset\Omega\) satisfy \eqref{eq:support_family_forward}. There exist constants
\(\gamma>0\), \(\Phi_{\max}>a_d\), and \(\sigma_{\min}>0\) and there exist representatives of \(u\) and \(\sigma\) such that, for every \(t\ge0\), the following inequalities hold for the relevant points \(x\in\mathcal K(t)\):
\begin{align}
\Phi(t,x)\le\Phi_{\max}, \qquad x\in\mathcal K(t),
\label{eq:Phi_upper_support}
\end{align}
\begin{align}
D_u\Phi(t,x)\le-\gamma, \qquad x\in\mathcal K(t)\setminus V_d(t), 
\label{eq:DuPhi_negative_outside}
\end{align}
and
\begin{align}
\sigma(t,x)\ge\sigma_{\min},\qquad x\in V_d(t)\cap\mathcal K(t), t\ge0.
\label{eq:sigma_positive_active_region}
\end{align}
\end{assumption}
Notice that \eqref{eq:DuPhi_negative_outside} is a Lyapunov-type entrance condition for the characteristic flow, ensuring that outside the active region \(V_d\), \(\Phi\) decreases uniformly along all relevant characteristics, in the same spirit as a Lyapunov decrease condition. This condition is stronger than the characteristic damping condition of Definition \ref{def:characteristic_damp}, which is the intrinsic geometric requirement for stabilization, but may be difficult to verify directly, since it involves the accumulated damping along the full characteristic flow. In contrast, Assumption \ref{ass:support_lyapunov}, as it is shown below, provides a constructive sufficient condition that guarantees that every relevant characteristic enters the active region and remains there, so that a positive amount of damping is accumulated over every sufficiently long time horizon. The nonlinear extension in Section \ref{sec:nonlinear} uses the same idea with a more general nonlinear Lyapunov-type entrance condition. 
\begin{proposition}
\label{prop:finite_time_entrance}
Suppose Assumption \ref{ass:support_lyapunov} holds. Then,  every characteristic starting in \(\mathcal K(t)\) enters \(V_d(s)\cap\mathcal K(s)\) within the uniform time
\[
T_{\mathrm{in}} := \frac{\Phi_{\max}-a_d}{\gamma}.
\]
More precisely, for every \(t\ge0\) and every \(x\in\mathcal K(t)\),
\begin{align}
X(s;t,x)\in V_d(s)\cap\mathcal K(s), \qquad \forall s\ge t+T_{\mathrm{in}}.
\label{eq:finite_time_entrance}
\end{align}
Consequently, for every \(T>T_{\mathrm{in}}\),
\begin{align}
&\int_t^{t+T}\sigma(s,X(s;t,x)) ds \ge
\sigma_{\min}(T-T_{\mathrm{in}}).
\label{eq:entrance_implies_damping}
\end{align}
\end{proposition}

\begin{proof}
Since \(X(\cdot;t,x)\) is absolutely continuous and \(\Phi\) is in \(C^1\), the map \(s\mapsto\Phi(s,X(s;t,x))\) is absolutely continuous and
\[
\frac{d}{ds}\Phi(s,X(s;t,x))=D_u\Phi(s,X(s;t,x))
\]
for a.e. \(s\). Fix \(t\ge0\) and \(x\in\mathcal K(t)\), and set \(y(s):=\Phi(s,X(s;t,x)).\)  Since \(\mathcal K(t)\) is forward invariant, \(X(s;t,x)\in\mathcal K(s)\) for \(s\ge t\). Whenever
\(y(s)>a_d\), the characteristic lies outside \(V_d(s)\), so
\[
y'(s)=D_u\Phi(s,X(s;t,x))\le-\gamma.
\]
Since \(y(t)\le\Phi_{\max}\), the value \(y(s)\) reaches the level \(a_d\) no later than \(t+T_{\mathrm{in}}\). Moreover, once \(y\) has reached the sublevel set, it cannot exit it. Indeed, if \(y(s_0)\le a_d\) and \(y(s_1)>a_d\) for some \(s_1>s_0\), continuity gives a \(\tau\in[s_0,s_1)\) such that \(y(\tau)=a_d\) and \(y(s)>a_d\) on \((\tau,s_1]\). On this interval \(y'\le-\gamma\) a.e., hence, \(y(s_1)\le a_d-\gamma(s_1-\tau)<a_d\), a contradiction.
The damping estimate \eqref{eq:entrance_implies_damping} follows from the assumption \(\sigma\ge\sigma_{\min}\) on \(V_d(s)\cap\mathcal K(s)\) (\(\sigma\) is represented by a bounded Borel function and the lower bound holds pointwise on the relevant region). 
\end{proof}

\begin{remark}
Assumption \ref{ass:support_lyapunov} is sufficient to infer the finite-time entrance property in Proposition \ref{prop:finite_time_entrance}. In addition, it will be used directly in the weighted Lyapunov estimate of Section \ref{sec:stabilization} using the fact that the inequality \eqref{eq:DuPhi_negative_outside} guarantees decay of the weighted energy outside the active region, while the lower bound on \(\sigma\) gives damping inside the active region.
\end{remark}

\begin{remark}
The support restriction is essential under the impermeability boundary condition we assumed. Since the flow maps \(\Omega\) onto itself, an active region \(V_d(t)\Subset\Omega\) cannot contain all characteristics after finite time. Therefore, the entrance condition must be imposed either globally through recurrent visits, or only on the characteristics carrying the initial error \(e_0\) and perturbation \(w\).
\end{remark}

\begin{remark}
In the case of a static velocity field, Assumption \ref{ass:support_lyapunov} is verified by a standard potential construction. Suppose that velocity field \(u(x)\) is written as 
\[
u(x)=-a(x)\nabla\Phi(x),\qquad a(x)\ge a_{\min}>0,
\]
and let \(V_d=\{x\in\Omega:\Phi(x)\le a_d\}\). If, on
\(\mathcal K(t)\setminus V_d\),
\[|\nabla\Phi(x)|\ge g_{\min}>0,\]
then
\[
D_u\Phi(x)=-a(x)|\nabla\Phi(x)|^2 \le -a_{\min}g_{\min}^2.
\]
Thus, \eqref{eq:DuPhi_negative_outside} holds with \(\gamma=a_{\min}g_{\min}^2\). If \(\Phi\) is bounded above on \(\mathcal K(t)\) and \(\sigma\ge\sigma_{\min}>0\) on \(V_d\cap\mathcal K(t)\), then Proposition \ref{prop:finite_time_entrance} gives characteristic damping on the support family.
\end{remark}

\begin{remark}
It is worth mentioning here that by the characteristic damping condition we may also derive an observability estimate for the undamped transport equation
\[
\partial_t z+\nabla\cdot(uz)=0.
\]
Assume for simplicity the global condition \(\mathcal D_T(t,x)\ge m_T\). Then, the characteristic representation and Liouville's formula give
\[ \int_t^{t+T}\int_\Omega \sigma(s,y)z(s,y)^2 dy ds \ge m_Te^{-M^+T}\|z(t)\|_{L^2(\Omega)}^2, \]
where \( M^+:=\|(\nabla\cdot u)_+\|_{L^\infty(\mathbb R_{\ge 0}\times\Omega)}. \) This can be seen by invoking formula \(z(s,X(s;t,x))^2J(s;t,x)= z(t,x)^2e^{-\int_t^s \nabla\cdot u(\tau,X(\tau;t,x))d\tau}\), coming from the characteristic representation, and performing change of variables \(y=X(s;t,x)\).
In that sense, \(\mathcal D_T\ge m_T\) is the transport analogue of a geometric observability condition, see, for instance, \cite{Bardos1992,Lions1988,Zuazua2005}. Localized observations also arise in observer-based inverse problems for waves, see, for example, \cite{KitsosBajodekBaudouin2023}.
The characteristic damping condition is, therefore, an intrinsic geometric requirement since if relevant characteristics accumulate arbitrarily small damping over every fixed time horizon, then initial errors can be concentrated near such characteristics, and the localized damping has arbitrarily small effect. As a result, the derivation of a uniform localized stabilization estimate cannot be expected without positive accumulated damping on the relevant characteristic family.
\end{remark}

\section{Stabilization and ISS Estimates}
\label{sec:stabilization}

In this section, we prove stabilization of the closed-loop error dynamics. We first give a general result based only on the characteristic damping condition. We then present a weighted Lyapunov criterion based on the Lyapunov-type function \(\Phi\). The first result is more general, while the second derives a classical differential Lyapunov inequality and leads directly to ISS estimates.

\subsection{Stabilization from Finite-Time Characteristic Damping}\label{subsec:Stab_finite-time_ch}

Consider the unperturbed (\(w\equiv 0\)) closed-loop error equation
\begin{align}
\partial_t e+\nabla\cdot(ue)=-\kappa\sigma e.
\label{eq:unperturbed_error_stab}
\end{align}

Let \(\mathcal A(t)\subset\Omega\) denote either the whole domain \(\Omega\), in the global case, or a forward-invariant support family \(\mathcal K(t)\) generated as in Lemma \ref{lem:support_propagation}, where \(\mathcal K(t)\) satisfies \eqref{eq:support_family_forward} (or any measurable family containing that reachable support). Define
\begin{align}
M^-_{\mathcal A}:=\operatorname*{ess\,sup}_{t\ge0, x\in\mathcal A(t)} (\nabla\cdot u(t,x))_-. \label{eq:M_A}
\end{align}

\begin{theorem}
\label{thm:char_damping_stab}
Assume that the characteristic damping condition holds on \(\mathcal A(t)\), namely, there exist \(T>0\) and \(m_T>0\) such that \eqref{eq:finite_time_char_damping} holds true. In the support-restricted case \(\mathcal A(t)=\mathcal K(t)\), assume in addition that
\(\operatorname{ess\,supp}e_0\subset\mathcal K(0),\)
similarly to Lemma \ref{lem:support_propagation}. 
If gain condition
\begin{align}
2\kappa m_T>M^-_{\mathcal A}T,
\label{eq:char_gain_condition}
\end{align}
holds for sufficiently large feedback gain \( \kappa>0\), then the closed-loop equation \eqref{eq:unperturbed_error_stab} is exponentially stable in \(L^2(\Omega)\) for all initial errors supported in \(\mathcal A(0)\).  More precisely, for every solution with \(\operatorname{ess\,supp}e_0\subset\mathcal A(0)\), there exist constants
\(C_T>0\) and \(\alpha_T>0\), such that
\begin{align}
\|e(t)\|_{L^2(\Omega)}^2 \le C_Te^{ -\alpha_T t}\|e_0\|_{L^2(\Omega)}^2, \qquad t\ge0.
\label{eq:char_exp_decay}
\end{align}
We may take \(C_T=e^{M^-_{\mathcal A}T}, \alpha_T=\frac{2\kappa m_T}{T}-M^-_{\mathcal A}\).
\end{theorem}
\begin{remark}
Stabilization is obtained when the accumulated feedback damping dominates the compressive amplification of the transported error as expressed by the gain condition \eqref{eq:char_gain_condition}. The quantity \(m_T\), therein (see Definition \ref{def:characteristic_damp}), measures the amount of localized damping accumulated along a characteristic over the interval \([t,t+T]\) and \(M^-_{\mathcal A}T\) measures the worst possible \(L^2\)-growth caused by compression of the transport field on the same time interval. Thus, stabilization requires not only a sufficiently large gain, but also a geometric condition ensuring \(m_T>0\). In particular, if a relevant characteristic never encounters the region where \(\sigma>0\), then \(m_T=0\), and the finite-time contraction estimate cannot hold for any feedback  gain \( \kappa\).

It is worth noting that if one wishes to shrink the active region \(V_d\) where the feedback acts, this reduces \(m_T\) and requires a larger feedback gain \(\kappa\) in \eqref{eq:char_gain_condition}. In the present setting, where \(\sigma\in L^\infty(\Omega)\), an active set of lower dimension, such as a curve or a hypersurface instead of a set of positive measure, would generally yield zero accumulated damping for characteristics crossing it. Such singular actuation is outside of the scope of the present framework.
\end{remark}
\begin{proof}[Proof of Theorem \ref{thm:char_damping_stab}]
We give the proof when the solution is confined to \(\mathcal{K}(t)\)  (the global case follows by
taking \(\mathcal A(t)=\Omega\)). By Lemma \ref{lem:support_propagation}, the
unperturbed solution remains supported in \(\mathcal A(t)\). For \(x\in\mathcal A(t)\), the characteristic representation and Liouville's formula
give
\begin{align*}
&e(t+T,X(t+T;t,x))^2J(t+T;t,x) \\
&= e(t,x)^2 e^{-\int_t^{t+T}\nabla\cdot u(s,X(s;t,x)) ds
-2\kappa\int_t^{t+T}\sigma(s,X(s;t,x))ds}.
\end{align*}
Using
\(-\int_t^{t+T}\nabla\cdot u(s,X(s;t,x)) ds
\le M^-_{\mathcal A}T\)
and \eqref{eq:finite_time_char_damping}, we obtain
\[
e(t+T,X(t+T;t,x))^2J(t+T;t,x) \le e^{M^-_{\mathcal A}T-2\kappa m_T}e(t,x)^2.
\]
Changing variables \(y=X(t+T;t,x)\) yields
\[
\|e(t+T)\|_{L^2(\Omega)}^2 \le q_T\|e(t)\|_{L^2(\Omega)}^2, \qquad q_T:=e^{M^-_{\mathcal A}T-2\kappa m_T}.
\]
We have \(q_T<1\) by virtue of \eqref{eq:char_gain_condition}. Iterating over intervals of length \(T\)  (writing \(t=nT+r\), \(r\in[0,T)\)), and using the rough bound (this bound follows from the same characteristic formula with \(\sigma\ge0\), using the propagation on \(\mathcal A(t)\))
\[
\|e(t+r)\|_{L^2(\Omega)}^2 \le e^{M^-_{\mathcal A}r}\|e(t)\|_{L^2(\Omega)}^2,
\qquad 0\le r<T,
\]
gives \eqref{eq:char_exp_decay} with \(\alpha_T=-T^{-1}\log q_T\) and a suitable \(C_T\).  
\end{proof}

\begin{remark}
The proof of Theorem \ref{thm:char_damping_stab} invokes finite-time contraction over intervals of length \(T\), followed by iteration. This method is in the spirit of classical links between integral/finite-time decay estimates and uniform exponential stability as in \cite{Datko1972}, although here the contraction is obtained directly from the characteristic formula.
\end{remark}

\begin{corollary}\label{cor:stab_from_entrance}
Assume that the hypotheses of Proposition \ref{prop:finite_time_entrance} hold
on a forward-invariant family \(\mathcal K(t)\). Then, for every \(T>T_{\mathrm{in}}\), the characteristic damping condition holds on \(\mathcal K(t)\) with
\(m_T=\sigma_{\min}(T-T_{\mathrm{in}}).\)
Consequently, if
\begin{align}
2\kappa\sigma_{\min}(T-T_{\mathrm{in}}) > M^-_{\mathcal K}T
\label{eq:entrance_gain_condition}
\end{align}
for some \(T>T_{\mathrm{in}}\), then the unperturbed closed-loop equation is exponentially stable for all initial errors supported in \(\mathcal K(0)\). In addition, such a \(T\) exists whenever
\begin{align}
2\kappa\sigma_{\min}>M^-_{\mathcal K}.
\label{eq:entrance_asymptotic_gain}
\end{align}
\end{corollary}
\begin{proof}
By Proposition \ref{prop:finite_time_entrance}, inequality \eqref{eq:entrance_implies_damping} holds
for \(x\in\mathcal K(t)\). The result follows from Theorem \ref{thm:char_damping_stab}. Finally,
noticing that \(\frac{T}{T-T_{\mathrm{in}}}\to 1 \text{ as }T\to\infty,\) \eqref{eq:entrance_gain_condition} is feasible for some \(T>T_{\mathrm{in}}\) whenever \eqref{eq:entrance_asymptotic_gain} holds.
\end{proof}

\subsection{A Weighted Lyapunov Criterion}

The main theorem of Section \ref{subsec:Stab_finite-time_ch} is the most general stabilization approach used here, since it requires only accumulated damping along the relevant characteristics. We now present a stronger differential Lyapunov criterion, which is useful when the Lyapunov-type entrance function \(\Phi\) itself can serve as an energy weight.

Let \(\mathcal K(t)\subset\Omega\) be a forward-invariant support family and assume again \(\operatorname{ess\,supp}e_0\subset\mathcal K(0).\)
Assume, in addition to \(\Phi\in C^1(\mathbb R_{\ge 0}\times\overline\Omega)\) in the previous section, that \(\Phi\) is uniformly bounded on \(\mathcal K(t)\), namely that there exist constants \(\Phi_{\min}\), \(\Phi_{\max} \in\mathbb R\), such that 
\begin{align} 
\Phi_{\min}\le\Phi(t,x)\le\Phi_{\max}, \qquad x\in\mathcal K(t), t\ge0 \label{eqPhi_bounded}
\end{align}
(this condition is automatic, for instance, when \(\Phi\) is time-independent, since \(\overline\Omega\) is compact).
For \(\ell>0\), define Lyapunov functional
\begin{align}
W_\ell(t):=\frac12\int_\Omega e^{\ell\Phi(t,x)}e(t,x)^2 dx.
\label{eq:weighted_lyap_functional}
\end{align}
Weighted Lyapunov functionals are classical for one-dimensional hyperbolic systems, see, e.g., \cite{BastinCoron2016}. Related weighted energy ideas also appear in stabilization of dispersive PDEs \cite{KitsosCerpa2021}. In the approach here, the weight is instead defined through the Lyapunov-type function \(\Phi\), so that characteristic decrease of \(\Phi\) outside \(V_d\) yields decay of the energy.

For solutions supported in \(\mathcal K(t)\), this gives \[\frac12e^{\ell\Phi_{\min}}\|e(t)\|_{L^2(\Omega)}^2 \le W_\ell(t)\le \frac12e^{\ell\Phi_{\max}}\|e(t)\|_{L^2(\Omega)}^2. \]
On the support of \(e\), this functional is equivalent to the usual \(L^2\)-energy.

Define \(L_{\mathcal K}:=\operatorname*{ess\,sup}_{t\ge0, x\in\mathcal K(t)\cap V_d(t)} (D_u\Phi(t,x))_+.\) Note that \(L_{\mathcal K}<\infty\) holds, for instance, if \(\partial_t\Phi\), \(\nabla\Phi\), and \(u\) are uniformly bounded on the relevant set.

\begin{theorem}
\label{thm:weighted_lyap_stab}
Suppose that all conditions of Assumption \ref{ass:support_lyapunov} are satisfied in addition to uniform boundedness of \(\Phi\) as in \eqref{eqPhi_bounded}. Assume also that \(L_{\mathcal K}<+\infty \). Select \(\ell>0\), such that
\begin{align}
\ell\gamma>M^-_{\mathcal K},
\label{eq:weighted_outside_condition}
\end{align}
recalling that \(M^-_{\mathcal K}\) is as in \eqref{eq:M_A} and \(\gamma\) is as in \eqref{eq:DuPhi_negative_outside}, and then choose feedback gain \(\kappa>0\), such that
\begin{align}
2\kappa\sigma_{\min}> \ell L_{\mathcal K}+M^-_{\mathcal K}.
\label{eq:weighted_inside_condition}
\end{align}
Then, the closed-loop error equation \eqref{eq:unperturbed_error_stab} is exponentially stable in \(L^2(\Omega)\) for all initial errors supported in
\(\mathcal K(0)\). More precisely, for every
\[0<\alpha<\frac12\min\left\{\ell\gamma-M^-_{\mathcal K}, 2\kappa\sigma_{\min}-\ell L_{\mathcal K}-M^-_{\mathcal K}
\right\},\]
there exists \(C_\ell>0\) such that
\begin{align}
\|e(t)\|_{L^2(\Omega)}^2\le C_\ell e^{-2\alpha t}\|e_0\|_{L^2(\Omega)}^2, \qquad t\ge0.
\label{eq:weighted_exp_decay}
\end{align}
\end{theorem}

\begin{proof}
For smooth solutions, differentiating \eqref{eq:weighted_lyap_functional}, using
\[\partial_t e+u\cdot\nabla e=-(\nabla\cdot u)e-\kappa\sigma e,\]
and integrating by parts with \(u\cdot n=0\), gives
\begin{align*}
\dot W_\ell (t)=
\frac12 \int_\Omega e^{\ell\Phi} \left(
\ell D_u\Phi-\nabla\cdot u-2\kappa\sigma \right)e^2 dx.
\end{align*}
The calculation is first justified for smooth data and coefficients. The general case follows by standard regularization and stability of the evolution family (alternatively, we may interpret the inequality in the sense of distributions in time).

On \(\mathcal K(t)\setminus V_d(t)\), we have
\[ \ell D_u\Phi-\nabla\cdot u-2\kappa\sigma \le -\ell\gamma+M^-_{\mathcal K}. \]
On \(\mathcal K(t)\cap V_d(t)\), we have
\[\ell D_u\Phi-\nabla\cdot u-2\kappa\sigma \le \ell L_{\mathcal K}+M^-_{\mathcal K}-2\kappa\sigma_{\min}. \]
By the choice of \(\alpha\), both right-hand sides are bounded above by \(-2\alpha\). Since the unperturbed solution is supported in \(\mathcal K(t)\),
\[
\dot W_\ell (t)\le -2\alpha W_\ell (t).
\]
Gr\"onwall's inequality yields
\(W_\ell(t)\le e^{-2\alpha t}W_\ell (0).\)
Since \(e^{\ell\Phi}\) is uniformly bounded above and below on
\(\mathcal K(t)\), \(W_\ell\) is equivalent to the usual \(L^2\)-energy on the support of the solution, whence we infer energy estimate \eqref{eq:weighted_exp_decay}.
\end{proof}

\begin{remark} \label{rem:monoton_navig}
If we further assume \(D_u\Phi(t,x)\le 0\) on \(\mathcal{K}(t)\cap V_d(t)\), for all \(t\ge0\), then \(L_{\mathcal{K}}=0\) and condition \eqref{eq:weighted_inside_condition} reduces to \(2\kappa\sigma_{\min}>M_{\mathcal{K}}^-\).
Choosing any \(\ell>M_{\mathcal{K}}^-/\gamma\), the functional \(W_\ell\) then yields exponential stability whenever the feedback gain dominates the compressive part of the velocity field on the relevant support family.
\end{remark}

\begin{remark}
Theorem \ref{thm:weighted_lyap_stab} is a stronger, pointwise criterion than Theorem \ref{thm:char_damping_stab}. The latter only requires accumulated damping over a horizon and it can cover moving or recurrently visited active regions. Theorem \ref{thm:weighted_lyap_stab} requires the function \(\Phi\) to provide instantaneous decay of the weighted energy outside the active region and tts advantage is that it gives a classical differential Lyapunov inequality.
\end{remark}

We finally provide an ISS estimate for 
\begin{align}
\partial_t e+\nabla\cdot(ue)=-\kappa\sigma e+w
\label{eq:perturbed_error_iss}
\end{align}
through the weighted Lyapunov functional introduced above.

\begin{theorem}
\label{thm:weighted_ISS}
Assume the hypotheses of Theorem \ref{thm:weighted_lyap_stab}. In the
support-restricted case, assume in addition that
\[
\operatorname{ess\,supp}w(t,\cdot)\subset\mathcal K(t), \qquad\text{for a.e. }t\ge0.
\]
Let \(\alpha\) be any constant satisfying the restriction in Theorem \ref{thm:weighted_lyap_stab}.
Then, there exist constants \(C_1,C_2>0\), such that every solution satisfies
\begin{align}
\|e(t)\|_{L^2(\Omega)}^2&\le C_1 e^{-\alpha t}\|e_0\|_{L^2(\Omega)}^2\notag\\&+ C_2\int_0^t e^{-\alpha(t-s)} \|w(s)\|_{L^2(\Omega)}^2 ds.\label{eq:weighted_ISS_estimate}
\end{align}
\end{theorem}
\begin{proof}
Under the hypotheses of Theorem \ref{thm:weighted_lyap_stab}, we get
\[
\dot W_\ell(t)\le-2\alpha W_\ell(t)+\int_\Omega e^{\ell\Phi(t,x)}e(t,x)w(t,x) dx.
\]
By Young's inequality and the equivalence between \(W_\ell\) and \(\|e\|^2_{L^2(\Omega)}\) on \(\mathcal K(t)\), there exists \(C>0\), such that \(\int_\Omega e^{\ell\Phi}e w dx \le \alpha W_\ell(t)+C\|w(t)\|_{L^2(\Omega)}^2.\)
Hence,
\[ \dot W_\ell(t) \le -\alpha W_\ell(t)+C\|w(t)\|_{L^2(\Omega)}^2.
\]
The equivalence between \(W_\ell\) and \(\|e\|^2_{L^2(\Omega)}\), in conjunction with Gr\"onwall's inequality, yields \eqref{eq:weighted_ISS_estimate}.

\end{proof}

\section{Realization of the Localized Feedback Flux \(J_d\)}
\label{sec:flux}

The stability results above were derived after prescribing the divergence of the control flux. In this section, we discuss realizations of a flux \(J_d\) satisfying
\[ \nabla\cdot J_d=F,\qquad F:=\kappa\sigma e-r^\star,\qquad r^\star:=\nabla\cdot(\rho^\star u),\]
which is the relevant quantity in the balance law, since the density dynamics depend on \(J_d\) only through \(\nabla\cdot J_d\). We can see that \(J_d\) is not unique, since adding any divergence-free vector field leaves the closed-loop error equation unchanged. The construction of \(J_d\) is, therefore, a right-inverse problem for the divergence operator. Continuous right inverses for the divergence are classical, see \cite{Bogovskii1979}.

If the construction is performed on the whole domain, we may solve, for each fixed time \(t\),
\begin{align}
\begin{aligned}
-\Delta\eta(t,\cdot)&=F(t,\cdot) \qquad\text{in }\Omega,
\\ \eta(t,\cdot)&=0 \qquad\text{on }\partial\Omega,
\label{eq:global_poisson_flux}\end{aligned}
\end{align}
and set
\[J_d(t,\cdot):=-\nabla\eta(t,\cdot). \]
Then, \(\nabla\cdot J_d=F\) in \(\Omega\), in the sense of distributions, and the closed-loop error equation \eqref{eq:error_closed_loop} is obtained exactly. Although \(F\) is localized, the corresponding flux \(J_d\) need not be compactly supported, because elliptic right inverses are generally nonlocal. Thus, the global problem should be seen as a theoretical right-inverse construction of the prescribed divergence. A more local actuator representation can be obtained by imposing the divergence law on an interior actuator region and prescribing the corresponding interface flux, as described next.

To obtain such a local representation, let \(U_d\Subset\Omega\) be an open Lipschitz region such that
\[\operatorname{ess\,supp}\sigma(t,\cdot)\cup \operatorname{ess\,supp}r^\star(t,\cdot)\Subset U_d. \]
Recall that \(V_d\) is the active damping region used in the stability proof, and we define here \(U_d\) as the region where the feedback flux is constructed. Typically, one takes \(V_d\Subset U_d\Subset\Omega\), where \(V_d\) is the region on which \(\sigma\) is uniformly positive and \(U_d\) contains the support of \(\sigma\) and of \(r^\star\). If, in a particular design, \(\operatorname{ess\,supp}\sigma\cup\operatorname{ess\,supp}r^\star\Subset \operatorname{int}(V_d)\), then one may take \(U_d=\operatorname{int}(V_d)\).

Since \(F\) is supported in \(U_d\), the nontrivial part of the divergence law is imposed in the actuated region, namely,
\begin{align}
\nabla\cdot J_d(t,\cdot)=F(t,\cdot)
\qquad\text{in }U_d.
\label{eq:local_divergence_flux}
\end{align}
To determine one representative of \(J_d\), let us prescribe an interface flux
\[
J_d(t,\cdot)\cdot n_{U_d}=g_d(t,\cdot) \qquad\text{on }\partial U_d,
\label{eq:interface_flux_condition}
\]
where
\(g_d(t,\cdot)\in H^{-1/2}(\partial U_d)\) satisfies the compatibility condition
\begin{align}
\int_{\partial U_d}g_d(t,x) dS(x)= \int_{U_d}F(t,x)\,dx.
\label{eq:interface_compatibility}
\end{align}
The compatibility condition \eqref{eq:interface_compatibility} is necessary by the divergence theorem, and it allows the actuated region to exchange density with the surrounding domain through its internal boundary. For example, one may choose
\[g_d(t,x)=\chi_d(x)\int_{U_d}F(t,y) dy,\qquad
\int_{\partial U_d}\chi_d(x) dS(x)=1, \]
with a nonnegative \(\chi_d\in L^\infty(\partial U_d)\). Since \(F(t,\cdot)\in L^2(U_d)\), the scalar \(\int_{U_d}F(t,y)\,dy\) is finite, and, since \(\partial U_d\) has finite surface measure, this choice gives \(g_d(t,\cdot)\in L^\infty(\partial U_d)
\subset L^2(\partial U_d) \hookrightarrow H^{-1/2}(\partial U_d).\)

A convenient construction is obtained through a Neumann potential problem. We choose \(g_d(t,\cdot)\in H^{-1/2}(\partial U_d)\) satisfying \eqref{eq:interface_compatibility} and solve
\begin{align}
\begin{aligned}
-\Delta\eta(t,\cdot)&=F(t,\cdot) \qquad\text{in }U_d,
\\ -\nabla\eta(t,\cdot)\cdot n_{U_d}&=g_d(t,\cdot) \qquad\text{on }\partial U_d,
\label{eq:local_neumann_potential}
\end{aligned}
\end{align}
with the normalization
\[
\int_{U_d}\eta(t,x)\,dx=0,
\]
and set
\(J_d(t,\cdot):=-\nabla\eta(t,\cdot)\text{ in }U_d.\) If \(F(t,\cdot)\in L^2(U_d)\), \(g_d(t,\cdot)\in H^{-1/2}(\partial U_d)\), and \eqref{eq:interface_compatibility} holds, then the Neumann problem admits a weak solution, unique in \(H^1(U_d)/\mathbb R\) (unique up to an additive constant) and the zero-mean condition fixes a unique representative in \(H^1(U_d)\). Consequently, \(\nabla\cdot J_d=F\) in \(U_d\) and \(J_d\cdot n_{U_d}=g_d\) on \(\partial U_d\) in the weak sense. The boundary datum \(g_d\) represents the interface exchange generated by the localized actuator. Notice here that the special choice \(g_d=0\) is possible only when \(\int_{U_d}F(t,x) dx=0\) (by virtue of the divergence theorem), which is not satisfied in general for \(F=\kappa\sigma e-r^\star\).

Note that the local construction specifies one representative of the actuator flux inside \(U_d\). If we wished to extend \(J_d\) by zero outside \(U_d\), we would introduce an additional distributional interface term on \(\partial U_d\). In the present formulation, the exact closed-loop equation is determined by the global divergence law \(\nabla\cdot J_d=F\) on \(\Omega\), while the local problem describes how a flux producing the prescribed localized divergence may be computed inside the actuated region. As a result, we obtain localization on the feedback divergence \(F\),  and not necessarily on the vector field \(J_d\) itself.

\section{Extension to Density-Dependent Velocity Fields}
\label{sec:nonlinear}

The purpose of this section is to extend the geometric entrance condition of Section \ref{sec:geometry} to velocity fields depending (nonlinearly) on the density. We show that the characteristic damping argument extends to density-dependent perturbations of the velocity field, provided the resulting flows belong to an admissible class satisfying uniform bounds.

Assume that the velocity field is written as
\begin{align}
u[\rho](t,x)=u_0(t,x)+\varepsilon b[\rho](t,x), \label{eq:density_dependent_navigation_general}
\end{align}
where \(u_0\) is a nominal field, \(b[\rho]\) is a possibly nonlocal perturbation and \(\varepsilon\ge0\) measures the size of the perturbation (not necessarily small).  In this section, we take perturbation \(w\equiv 0\) for simplicity. 

Since the velocity field now depends on \(\rho\), the corresponding feedforward residual is nonlinear. We define, similarly to \eqref{eq:rho_star} for the linear case,
\[
r^\star[\rho](t,x):=\nabla\cdot(\rho^\star u[\rho](t,x))
\]
and impose the feedback law through the divergence constraint
\begin{align}
\nabla\cdot J_d=\kappa\sigma e-r^\star[\rho],
\label{eq:nonlinear_divergence_feedback}
\end{align}
in analogy to \eqref{eq:div_law} for the linear case.
Then, with \(e=\rho-\rho^\star\), the nonlinear closed-loop error equation is
\begin{align}
\partial_t e+\nabla\cdot(eu[\rho])=-\kappa\sigma e
\label{eq:nonlinear_error_equation}
\end{align} 
(noting \( u[\rho]=u[e+\rho^\ast]\)).
\begin{assumption}
\label{ass:nonlinear_admissible}
Let \(\mathfrak A\) be a class of admissible densities \(\rho\) such that, for each \(\rho\in\mathfrak A\), the field \(u[\rho]\) given by \eqref{eq:density_dependent_navigation_general} belongs to
\(L^\infty_{\mathrm{loc}}\bigl(\mathbb R_{\ge 0};W^{1,\infty}(\Omega;\mathbb R^m)\bigr),\) with \(\nabla\cdot u[\rho]\in L^\infty(\mathbb R_{\ge 0}\times\Omega),\)
satisfies the boundary condition
\[u[\rho](t,\cdot)\cdot n=0\qquad\text{on }\partial\Omega,\]
and generates a Carath\'eodory flow
\begin{align}
\frac{d}{ds}X_\rho(s;t,x)=u[\rho](s,X_\rho(s;t,x)),\quad X_\rho(t;t,x)=x.
\label{eq:nonlinear_characteristic_flow}
\end{align}
We assume that the corresponding error \(e\) satisfies
\eqref{eq:nonlinear_error_equation} in the characteristic sense associated with \(X_\rho\).  We also assume that there exists a measurable family \(\mathcal K(t)\), forward invariant for all such flows, namely,
\begin{align}X_\rho(t;s,\mathcal K(s))\subset\mathcal K(t), \qquad t\ge s\ge0, \quad \rho\in\mathfrak A. \label{eq:nonl_forw_inv}
\end{align}
\end{assumption}

Similarly to Section \ref{sec:geometry}, we use a function \(\Phi\in C^1(\mathbb R_{\ge 0}\times\overline\Omega)\) to define the active region \(V_d(t):=\{x\in\Omega:\Phi(t,x)\le a_d\},\)
where \(a_d\in\mathbb R\) is chosen so that \(V_d(t)\) is nonempty for all \(t\ge 0\). We write
\[
D_{u[\rho]}\Phi(t,x):=\partial_t\Phi(t,x)+u[\rho](t,x)\cdot\nabla\Phi(t,x)
\]
for the case of the density-dependent field \(u[\rho].\)

We now state the nonlinear analogue of Assumption \ref{ass:support_lyapunov} of Section \ref{sec:geometry}.
\begin{assumption}
\label{ass:nonlinear_entr}
There exist \(\Phi_{\max}>a_d\), \(\sigma_{\min}>0\), and a function \(\beta\in C(\mathbb R_{\ge0};\mathbb R_{\ge0})\), with \(\beta(r)>0\), for every \(r>0\),  satisfying the following integrability condition: 
\begin{align}
T^\ast:=\int_0^{\Phi_{\max}-a_d}\frac{dr}{\beta(r)}<+\infty.
\label{eq:alpha_condition}
\end{align}
Moreover, for every admissible density \(\rho\) in \(\mathfrak A\), we have \(\Phi(t,x)\le\Phi_{\max}, x\in\mathcal K(t), t\ge 0\) and 
\begin{align}
D_{u[\rho]}\Phi(t,x) \le -\beta(\Phi(t,x)-a_d), \quad x\in\mathcal K(t)\setminus V_d(t). \label{eq:nonlinear_lyap}
\end{align}
Finally, the localization function satisfies
\begin{align}
\sigma(t,x)\ge\sigma_{\min}, \qquad x\in V_d(t)\cap\mathcal K(t),
\label{eq:nonlinear_sigma_min}
\end{align}
for a.e. \(t\ge0\).
\end{assumption}
Condition \eqref{eq:nonlinear_lyap} is a nonlinear Lyapunov-type entrance condition. Along a characteristic of the density-dependent field \(u[\rho]\), the quantity
\begin{align}
y(s):=\Phi(s,X_\rho(s;t,x))-a_d \label{eq:y(s)}
\end{align}
satisfies a.e. the comparison inequality \[\dot y(s)\le -\beta(y(s))\] whenever \(y(s)>0\), as a result of the fact that \(\dot y(s)=D_{u[\rho]}\Phi(s,X_\rho(s;t,x))\).

We are now in a position to state a result in the spirit of Proposition \ref{prop:finite_time_entrance} providing a uniform entrance time for all admissible nonlinear flows.
\begin{proposition}
\label{prop:nonlinear_finite_time_entrance}
Suppose Assumption \ref{ass:nonlinear_entr} holds. Then, every admissible characteristic starting in \(\mathcal K(t)\) enters \(V_d(s)\cap\mathcal K(s)\) within the uniform time
\begin{align}
T_{\mathrm{in}}:=T^\ast,
\label{eq:Tin_alpha}
\end{align}
with \(T^\ast\) given by \eqref{eq:alpha_condition}. More precisely, for every density \(\rho\) in \(\mathfrak A\), every \(t\ge0\), and every relevant point \(x\in\mathcal K(t)\),
\begin{align}
X_\rho(s;t,x)\in V_d(s)\cap\mathcal K(s),
\qquad s\ge t+T_{\mathrm{in}}.
\label{eq:nonlinear_finite_time_entrance}
\end{align}
Consequently, for every \(T>T_{\mathrm{in}}\),
\begin{align}
\int_t^{t+T}\sigma(s,X_\rho(s;t,x)) ds \ge \sigma_{\min}(T-T_{\mathrm{in}}).
\label{eq:nonl_damp_beta}
\end{align}
\end{proposition}

\begin{proof}
For an admissible density \(\rho\), \(t\ge0\), and \(x\in\mathcal K(t)\), consider \(y\) as in \eqref{eq:y(s)}. By \eqref{eq:nonl_forw_inv}, \(X_\rho(s;t,x)\in\mathcal K(s)\) for \(s\ge t\). Whenever \(y(s)>0\), the characteristic lies outside \(V_d(s)\), and, therefore,
\(y'(s) \le -\beta(y(s))\)
for a.e. such \(s\). Since \(y(t)\le \Phi_{\max}-a_d\), the comparison estimate ensures entrance into \(\{y\le0\}\) no later than \(T_{\mathrm{in}}\). After \(T_{\mathrm{in}}\), the characteristic cannot leave \(V_d\). Indeed, if \(y(s_0)\le0\) and \(y(s_1)>0\) for some \(s_1>s_0\), then by continuity there exists \(\tau\in[s_0,s_1)\) such that \(y(\tau)=0\) and \(y(s)>0\) on \((\tau,s_1]\). On this interval \(y'\le-\beta(y)<0\), contradicting \(y(s_1)>0\). Hence, \eqref{eq:nonlinear_finite_time_entrance} holds and the estimate \eqref{eq:nonl_damp_beta} follows from \eqref{eq:nonlinear_sigma_min}.
\end{proof}
Note that the less general condition \(D_u\Phi\le-\gamma\) we used in Section \ref{sec:geometry} is recovered from \eqref{eq:nonlinear_lyap} by taking \(\beta\equiv\gamma\). The nonlinear decay rate coming from \(\beta(\Phi-a_d)\) is useful for nonlinear perturbations, where the inward field may depend on the distance to the active sublevel set. We utilize an additional integrability condition \eqref{eq:alpha_condition} to meet the finite-time entrance requirement for this comparison inequality. For instance, taking \(\beta(r)=cr\) ensures only asymptotic entrance, whereas \(\beta(r)=cr^p\), \(0\le p<1\), ensures finite-time entrance.

To verify Lyapunov-type condition \eqref{eq:nonlinear_lyap}, we may utilize the following simpler criterion.
\begin{lemma}  \label{lemm:verify_nonlinear}
Assume that the nominal field \(u_0(x)\) satisfies
\begin{align}
D_{u_0}\Phi(t,x)\le-\beta_0(\Phi(t,x)-a_d),\quad x\in\mathcal K(t)\setminus V_d(t), \label{eq:D_u_0_ineq}
\end{align}
where \(\beta_0\in C(\mathbb R_{\ge0};\mathbb R_{\ge0})\) and \(\beta_0(r)>0\), for \(r>0\). Suppose that \eqref{eq:alpha_condition} holds and, for every density \(\rho\in\mathfrak A\), the nonlinear perturbation of the velocity field \eqref{eq:density_dependent_navigation_general} satisfies
\begin{align}
&\varepsilon b[\rho](t,x)\cdot\nabla\Phi(t,x) \notag\\\quad &\le \beta_0(\Phi(t,x)-a_d) -\beta(\Phi(t,x)-a_d)
\label{eq:nonl_bound}
\end{align}
on \(\mathcal K(t)\setminus V_d(t)\), where \(\beta\) is of class \(C(\mathbb R_{\ge0};\mathbb R_{\ge0})\), such that \(0<\beta(r)\le\beta_0(r), \)  for  \(r>0\). Then, \(u[\rho]\) satisfies \eqref{eq:nonlinear_lyap} uniformly over all admissible densities and Proposition \ref{prop:nonlinear_finite_time_entrance} applies.
\end{lemma}
\begin{proof}
We readily obtain \eqref{eq:nonlinear_lyap} by noticing that \(D_{u[\rho]}\Phi=D_{u_0}\Phi+\varepsilon b[\rho]\cdot\nabla\Phi\), in conjunction with \eqref{eq:D_u_0_ineq} and \eqref{eq:nonl_bound}.
\end{proof}

Note that condition \eqref{eq:nonl_bound} controls only the component of the perturbation in the direction \(\nabla\Phi\). It is automatically satisfied by tangential components to the level sets of \(\Phi\) and inward components satisfying \(b[\rho]\cdot\nabla\Phi\le0\).

We also present a class of nonlocal perturbations, which is useful because the condition \eqref{eq:nonl_bound} can be verified directly from the kernel. Note here that although the map \(\rho\mapsto b[\rho]\) below is assumed linear, the closed-loop transport equation remains quasilinear. 
\begin{corollary}
\label{cor:nonlocal_kernel}
Let
\begin{align} b[\rho](t,x)=\int_\Omega B(x,y)\rho(t,y) dy, \label{eq:nonl_functional}\end{align}
where \(B\in L^\infty(\Omega\times\Omega;\mathbb R^m)\) with \(\nabla_x B\in L^\infty(\Omega\times\Omega;\mathbb R^{m\times m}).\) Assume also that the kernel satisfies the boundary compatibility condition
\[B(x,y)\cdot n(x)=0, \qquad x\in\partial\Omega,\quad y\in\Omega. \] Assume that admissible densities \(\rho \in \mathfrak A\), as in Assumption \ref{ass:nonlinear_admissible}, are nonnegative and satisfy \(\|\rho(t,\cdot)\|_{L^1(\Omega)}\le M\), for \(M>0.\) It follows that \(b[\rho](t,\cdot)\in W^{1,\infty}(\Omega;\mathbb R^m)\), uniformly over \(\rho\in\mathfrak A\), and \(b[\rho]\cdot n=0\) on \(\partial\Omega\). Consequently, if \(u_0\) satisfies the regularity and boundary assumptions of Assumption \ref{ass:domain_u}, then so does \(u[\rho]\).

Assume now that the nominal field satisfies inequality \eqref{eq:D_u_0_ineq}, for some \(\beta_0\in C(\mathbb R_{\ge0};\mathbb R_{\ge0})\) with \(\beta_0(r)>0\) for \(r>0\). Suppose that there exists a nonnegative function \(k\in C([0,\Phi_{\max}-a_d];\mathbb R_{\ge 0})\), such that
\begin{align}
\bigl(B(x,y)\cdot\nabla\Phi(t,x)\bigr)_+ \le k(\Phi(t,x)-a_d),
\label{eq:directional_kernel_bound}
\end{align}
for \(x\in\mathcal K(t)\setminus V_d(t)\) and \(y\in\Omega\). Define
\[ \beta(r):=\beta_0(r)-\varepsilon M k(r). \]
If \(\beta(r)>0\) on \((0,\Phi_{\max}-a_d]\) and \eqref{eq:alpha_condition} is satisfied, then the nonlinear field \(u[\rho]\) satisfies the finite-time entrance condition with \(T_{\mathrm{in}}\) given by \eqref{eq:Tin_alpha}.
\end{corollary}

\begin{proof}
Since \(B\in L^\infty(\Omega\times\Omega;\mathbb R^m)\) and \(\|\rho(t,\cdot)\|_{L^1(\Omega)}\le M\), we have
\[\|b[\rho](t,\cdot)\|_{L^\infty(\Omega)}\le M\|B\|_{L^\infty(\Omega\times\Omega)}\]
and \[\|\nabla_x b[\rho](t,\cdot)\|_{L^\infty(\Omega)} \le M\|\nabla_x B\|_{L^\infty(\Omega\times\Omega)}, \]
whence, \(b[\rho](t,\cdot)\in W^{1,\infty}(\Omega;\mathbb R^m)\), uniformly over admissible densities. The boundary compatibility gives \(b[\rho]\cdot n=0\) on \(\partial\Omega\), thereby, \(u[\rho]\cdot n=0\) whenever \(u_0\cdot n=0\).

For \(x\in\mathcal K(t)\setminus V_d(t)\), using \(\rho\ge0\) and \(\|\rho(t,\cdot)\|_{L^1(\Omega)}\le M\), we have
\begin{align*}
&\bigl(b[\rho](t,x)\cdot\nabla\Phi(t,x)\bigr)_+
\notag\\&\quad \le \int_\Omega \bigl(B(x,y)\cdot\nabla\Phi(t,x)\bigr)_+\rho(t,y) dy\\ &\quad \le M k(\Phi(t,x)-a_d).
\end{align*}
Therefore,
\[\varepsilon b[\rho](t,x)\cdot\nabla\Phi(t,x)\le \varepsilon M k(\Phi(t,x)-a_d).
\]
Combining this estimate with \eqref{eq:D_u_0_ineq}, gives
\begin{align}
&D_{u[\rho]}\Phi(t,x)\le-\beta_0(\Phi(t,x)-a_d)+\varepsilon M k(\Phi(t,x)-a_d)\notag\\&\quad = -\beta(\Phi(t,x)-a_d).
\end{align}
Hence, Assumption \ref{ass:nonlinear_entr} holds, and Proposition \ref{prop:nonlinear_finite_time_entrance} applies.
\end{proof}

We considered above for simplicity a linear operator \eqref{eq:nonl_functional}, but  nonlinear (nonlocal) operators can be treated similarly, provided they satisfy the admissibility requirements of Assumption \ref{ass:nonlinear_admissible} and the bound \eqref{eq:nonl_bound}.

We are now in a position to state a stabilization result in the spirit of Theorem \ref{thm:char_damping_stab} concerning the linear case. 
\begin{theorem}
\label{thm:nonlinear_stabilization}
Suppose Assumptions \ref{ass:nonlinear_admissible} and
\ref{ass:nonlinear_entr} hold. Assume also that 
\begin{align}
M^-_{\mathcal K,\mathfrak A}:=\sup_{\rho\in\mathfrak A}
\operatorname*{ess\,sup}_{t\ge0,\ x\in\mathcal K(t)}
(\nabla\cdot u[\rho](t,x))_-<+\infty,\label{eq:nonl_compr}
\end{align}
 for the admissibility class. If there exists \(T>T_{\mathrm{in}}\), with \(T_{\mathrm{in}}\) given by
\eqref{eq:Tin_alpha}, such that
\begin{align}
2\kappa\sigma_{\min}(T-T_{\mathrm{in}})> M^-_{\mathcal K,\mathfrak A}T,
\label{eq:nonlinear_gain_condition}
\end{align}
then, every admissible solution of \eqref{eq:nonlinear_error_equation}  with \(\operatorname{ess\,supp}e_0\subset\mathcal K(0)\) satisfies 
\begin{align}
\|e(t)\|_{L^2(\Omega)}^2 \le C_Te^{-\alpha_Tt}\|e_0\|_{L^2(\Omega)}^2, \qquad t\ge0,
\label{eq:nonlinear_exp_decay}
\end{align}
for some \(C_T,\alpha_T>0\), uniformly over \(\rho\in\mathfrak A\).
\end{theorem}

\begin{proof}
We fix an admissible trajectory \(\rho\in\mathfrak A\), then, along \(X_\rho\), the error equation holds in the characteristic sense and gives
\begin{align*}&\frac{d}{ds}e(s,X_\rho(s;t,x))\notag\\&\quad =-\bigl(\nabla\cdot u[\rho]+\kappa\sigma\bigr)(s,X_\rho(s;t,x))e(s,X_\rho(s;t,x)),
\end{align*}
whence, we get the implicit formula
\begin{align*}
&e(t+T,X_\rho(t+T;t,x)) \\&\quad = e(t,x)  e^{ -\int_t^{t+T}
(\nabla\cdot u[\rho]+\kappa\sigma)(s,X_\rho(s;t,x)) ds},
\end{align*}
satisfying
\begin{align*}
&e(t+T,X_\rho(t+T;t,x))^2J_\rho(t+T;t,x)\\ &\quad =e(t,x)^2
e^{-\int_t^{t+T}\left( \nabla\cdot u[\rho](s,X_\rho(s;t,x))
+2\kappa\sigma(s,X_\rho(s;t,x)) \right) ds},
\end{align*}
with 
\[J_\rho(t+T;t,x):=\exp\left(\int_t^{t+T}\nabla\cdot u[\rho](s,X_\rho(s;t,x)) ds \right).\]
Using \eqref{eq:nonl_compr} and \eqref{eq:nonl_damp_beta}, we get
\[e(t+T,X_\rho(t+T;t,x))^2J_\rho(t+T;t,x)\le q_T e(t,x)^2,\]
where
\(q_T:=\exp\left(M^-_{\mathcal K,\mathfrak A}T-
2\kappa\sigma_{\min}(T-T_{\mathrm{in}})\right).\) By virtue of \eqref{eq:nonlinear_gain_condition}, \(q_T<1\) and changing variables gives
\[\|e(t+T)\|_{L^2(\Omega)}^2\le q_T\|e(t)\|_{L^2(\Omega)}^2.\]
A rough estimate \(\|e(t+r)\|_{L^2(\Omega)}^2\le e^{M^-_{\mathcal K,\mathfrak A}r}\|e(t)\|_{L^2(\Omega)}^2,\) for \(0\le r<T,\) follows from the same characteristic formula with \(\sigma\ge0\) and after iterating over intervals of length \(T\), we obtain \eqref{eq:nonlinear_exp_decay}.
\end{proof}

\begin{remark}
The results of this section are conditional on the existence of an admissible class \(\mathfrak A\) of nonlinear flows. Within such a class, the characteristic flow depends on the evolving density, but the finite-time entrance estimate remains uniform. To establish well-posedness for a specific choice of the operator \(b[\rho]\) we can invoke standard fixed-point arguments under suitable Lipschitz assumptions, but this not needed for the conditional stability statement above and it is omitted here.
\end{remark}

\begin{remark}
If, in addition, we have \(\gamma_\beta:=\inf_{0<r\le \Phi_{\max}-a_d}\beta(r)>0,\) then, the weighted Lyapunov criterion of Theorem \ref{thm:weighted_lyap_stab} also applies with \(\gamma_\beta\) in place of \(\gamma\) and \(M^-_{\mathcal K,\mathfrak A}\) in place of \(M^-_{\mathcal K}\), provided \((D_{u[\rho]}\Phi)_+\) is uniformly bounded on \(\mathcal K(t)\cap V_d(t)\). If this infimum is zero, we may use Theorem \ref{thm:nonlinear_stabilization}.
\end{remark}

\FloatBarrier

\section{Illustrative Example}
\label{sec:example}

We illustrate the stabilization result of Section IV on the unit disk \(\Omega=\{x\in\mathbb R^2:\ |x|<1\}\). We consider an initial density \(\rho_0\) consisting of two bumps away from the target region, and a reference density \(\rho^\star\) supported in a small target region near the origin, as seen in Fig. \ref{fig:rho_star_e0}. The velocity field transports the density inwards, while the feedback damping acts only in a localized region around the target as in Fig. \ref{fig:geometry}. The purpose of this example is to verify, in a simple explicit setting, the support-restricted characteristic damping condition, the gain condition, the ISS estimate, and the localized flux realization.

\subsection{Example Problem Setup}

Let
\(0<r_\Sigma<r_d<r_\sigma<r_U<R_0<R_1<R_2<1\) and define the target and active regions by
\[\Sigma_d:=\{x:\ |x|<r_\Sigma\}, \qquad V_d:=\{x:\ |x| \leq r_d\}. \]
Recall that the target region \(\Sigma_d\) is where the reference profile \(\rho^\star\) is supported, while \(V_d\) is the region where the localization function \(\sigma\) is positive (damping is effective). The larger disk \[U_d:=\{x:\ |x|<r_U\}\] will be used below as the actuator region for the flux realization. The initial error \(e_0\) is supported inside \[K_0:=\{x:\ |x|\le R_0\}.\] For numerical illustration, we take
\( r_\Sigma=0.15,  r_d=0.30, r_\sigma=0.45,  r_U=0.55,\) \(R_0=0.75, R_1=0.85, R_2=0.95, \mu=1\).

The prescribed velocity field \(u\) is chosen as a static inward radial field, so that characteristics starting from the initial support move towards the target region. Let us choose a cut-off function \(q\in C^\infty([0,1];[0,1])\), such that
\(q(r)=1,  \text{for }0 \le r \le R_1,  q(r)=0,  \text{for }R_2\le r \le 1,\) and define the velocity field by
\[u(x):=-\mu q(|x|)x,\qquad \mu>0,
\]
satisfying \(u\in C^\infty(\overline\Omega;\mathbb R^2)\). Since \(q(1)=0\), we obtain boundary condition 
\[u(x)\cdot n(x)=0 \qquad\text{on }\partial\Omega.\]
The geometry is shown in Fig. \ref{fig:geometry}. 
\begin{figure}[t]
\centering
    \includegraphics[width=0.77\linewidth]{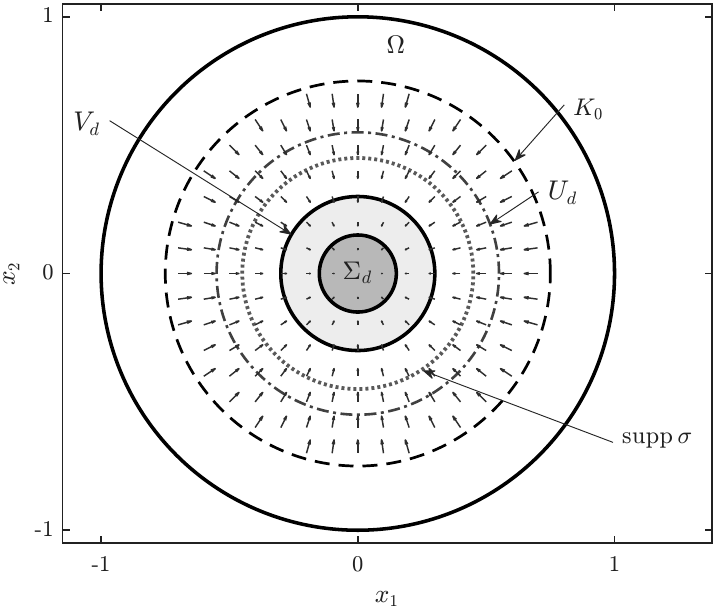}
\caption{Geometry of the example with nested circles \(\Sigma_d \Subset V_d \Subset \operatorname{supp}\sigma \Subset U_d \Subset K_0 \Subset \Omega\). The arrows depict the inward velocity field on the relevant support region.}
\label{fig:geometry}
\end{figure}

To design the reference profile and the initial density, we use smooth compactly supported bump functions with center \(x_i \in \mathbb R^2\) and radius $\varepsilon_i >0$ as 
\begin{align}
\psi_i(x):=
\begin{cases}
\exp\left(-\dfrac{\varepsilon_i^2}{\varepsilon_i^2-|x-x_i|^2}\right),
& |x-x_i|<\varepsilon_i,\\[2ex] 0,
& |x-x_i|\ge \varepsilon_i,
\end{cases} 
\label{eq:bump_functions}
\end{align}
and set, for \(A_1,A_2>0\),
\begin{align}
\rho_{A_1,A_2}(x):= \frac{\bigl(A_1\psi_1(x)+A_2\psi_2(x)\bigr)}
{\displaystyle\int_\Omega \bigl(A_1\psi_1(y)+A_2\psi_2(y)\bigr) dy}, \label{eq:rho_general}
\end{align}

The reference profile is chosen as a smooth density supported in \(\Sigma_d\), namely, \(\rho^\star(x):= \rho_{1,0.85}(x),\)
with \(\rho_{A_1,A_2}(x)\) as above, where bump functions \eqref{eq:bump_functions} are chosen with   \(x_1=(0.055,0.035),  x_2=(-0.050,-0.045), \) \(\varepsilon_1=\varepsilon_2=0.055.\) 
This choice satisfies
\(\rho^\star\in C_c^\infty(\Sigma_d),\) \(\rho^\star\ge0,\)
\( \int_\Omega\rho^\star(x) dx=1.\) The initial density is chosen as a sum of two bumps away from the target region \(\Sigma_d\) of the form \(\rho_0(x)=\rho_{1,0.9}(x),\)
with \(\rho_{A_1,A_2}(x)\) given by \eqref{eq:rho_general}, where bump functions \eqref{eq:bump_functions} are chosen with \(x_1=(0.55,0.16), x_2=(-0.47,-0.24), \varepsilon_1=\varepsilon_2=0.12\).
The positive part of \(e_0(x):=\rho_0(x)-\rho^\star(x)\) is supported in \(K_0\setminus V_d\), while its negative part is supported in \(\Sigma_d\subset V_d\), whence, \(\operatorname{supp}e_0\subset K_0\) as in Theorem \ref{thm:char_damping_stab}.

\begin{figure}[t]
\centering
\includegraphics[width=\linewidth]{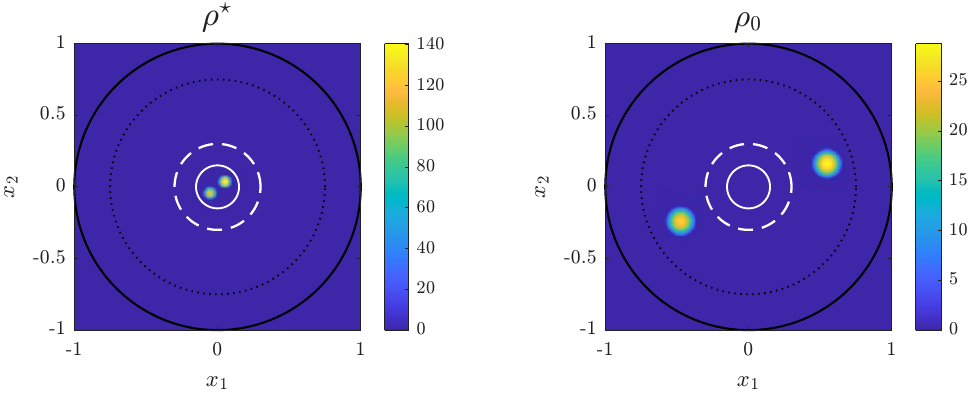}
\caption{Reference profile and initial density (note that the graphs use different color scales). The reference \(\rho^\star\) is supported in \(\Sigma_d\), while the initial density \(\rho_0\) consists of two smooth bumps supported in \(K_0\setminus V_d\). The dashed white circle denotes \(V_d\), the dotted black circle denotes \(K_0\), and the small solid white circle denotes \(\Sigma_d\).}
\label{fig:rho_star_e0}
\end{figure}

\FloatBarrier

\subsection{Controller Calculations}

We next verify the stabilization hypotheses. 

We see that since \(R_0<R_1\), every characteristic starting in \(K_0\) remains in the region where \(q(|x|)=1\), meaning that, on the relevant support family,
\(u(x)=-\mu x,\)
and the characteristic flow is explicitly given by
\(X(s;0,x)=e^{-\mu s}x, x\in K_0.
\)
The corresponding forward support family is
\[
\mathcal K(s):=X(s;0,K_0)=\{x:\ |x|\le e^{-\mu s}R_0\},
\] which is forward invariant.
The localization function \(\sigma\in C^\infty(\Omega;[0,1])\) is chosen radially such that \(\sigma(x)=1, |x|\le r_d, \sigma(x)=0,\quad |x|\ge r_\sigma.\)
Thus, \(\sigma=1\) on \(V_d\), and we may take \(\sigma_{\min}=1.\)
The support of the localization function is contained in the larger disk \(\{|x|\le r_\sigma\}\).

We verify the Lyapunov-type entrance condition using
\[\Phi(x):=|x|^2,\qquad a_d:=r_d^2,\]
so that the active damping region can be written as \(V_d=\{x\in\Omega:\ \Phi(x)\le a_d\}.\)
On the relevant support family, \(u(x)=-\mu x\), and, hence,
\( D_u\Phi(x)=u(x)\cdot\nabla\Phi(x) = (-\mu x)\cdot 2x = -2\mu |x|^2. \)
For \(x\in \mathcal K(t)\setminus V_d\), we have \(|x|\ge r_d\), so
\(D_u\Phi(x)\le -2\mu r_d^2.\)
Thus, Assumption \ref{ass:support_lyapunov} holds on \(\mathcal K(t)\) with
\(\gamma=2\mu r_d^2, \Phi_{\max}=R_0^2,
\sigma_{\min}=1.\)
Consequently, Proposition \ref{prop:finite_time_entrance} applies and gives the entrance time
\(\overline{T_{\rm in}}=\frac{R_0^2-r_d^2}{2\mu r_d^2}.\) Note that, by the explicit characteristic formula, we may use a sharper entrance time \(T_{\rm in}=\frac1\mu\log\frac{R_0}{r_d}.\)
Indeed, for \(x\in K_0\),
\(|X(s;0,x)|=e^{-\mu s}|x|\le e^{-\mu s}R_0,\)
so, \(X(s;0,x)\in V_d\), for all \(s\ge T_{\mathrm{in}}\). We calculate, for the numerical values above, \(T_{\rm in}=\log(2.5) \approx 0.9163.\) Now, since the vector field is assumed time-invariant, the same estimate holds for every initial time \(t\ge0\) and every \(x\in\mathcal K(t)\). Therefore, for every \(T>T_{\mathrm{in}}\),
\(\int_t^{t+T}\sigma(X(s;t,x)) ds \ge T-T_{\mathrm{in}},
\text{ for } x\in\mathcal K(t).\)
Thus, the characteristic damping condition holds on \(\mathcal K(t)\) with
\(m_T=T-T_{\mathrm{in}}.\)
On \(\mathcal K(t)\), \(u(x)=-\mu x\), and, therefore, \(\nabla\cdot u=-2\mu\), implying \(M^-_{\mathcal K}=2\mu.\)
All hypotheses of Corollary \ref{cor:stab_from_entrance} are, therefore, satisfied on the support family \(\mathcal K(t)\) and, thus, the closed-loop error equation is exponentially stable whenever there exists \(T>T_{\mathrm{in}}\), such that \(2\kappa(T-T_{\mathrm{in}})>2\mu T.\) 
It follows that for a sufficiently large horizon \(T\), this condition is feasible whenever \[\kappa>\mu,\] implying that the error satisfies an exponential decay estimate for all initial errors supported in \(K_0\). In the simulations, we use \(\kappa=2.5\), satisfying \(\kappa>\mu=1\).
The feedforward residual associated with the reference profile is
\(r^\star(x)=\nabla\cdot(\rho^\star(x)u(x)).
\)
Since \(\operatorname{supp}\rho^\star\subset \Sigma_d\subset\{|x|<R_1\}\), we have \(u(x)=-\mu x\) on the support of \(\rho^\star\) implying
\(r^\star(x)=\nabla\cdot(-\mu x\rho^\star(x))=-\mu x\cdot \nabla\rho^\star(x)-2\mu\rho^\star(x).\)
The feedback divergence \eqref{eq:div_law} is, therefore,
\(\nabla \cdot J_d=:F(t,x)=\kappa\sigma(x)e(t,x)-r^\star(x).\)
The term \(\kappa\sigma e\) is localized through \(\sigma\), while \(r^\star\) is localized in \(\Sigma_d\). Consequently,
\(\operatorname{ess\,supp}F(t,\cdot) \subset \operatorname{ess\,supp}\sigma\cup \operatorname{supp}r^\star
\subset\{x:\ |x|\le r_\sigma\}
\Subset U_d.\)

\subsection{Simulation Results and Discussion}

We now illustrate the closed-loop behavior. The density snapshots are computed using
\(\rho(t,x)=\rho^\star(x)+e(t,x),\)
where \(e\) solves the closed-loop error equation
\(\partial_t e+\nabla\cdot(ue)=-\kappa\sigma e.\) Using the explicit characteristic representation, we compute the density at the times
\(t=0, 0.5T_{\rm in}, T_{\rm in}, 1.5T_{\rm in}, 2T_{\rm in}, 3T_{\rm in}.\)
The snapshots are shown in Fig. \ref{fig:error_snap}. 

\begin{figure*}[t]
\centering
\includegraphics[width=0.92\textwidth]{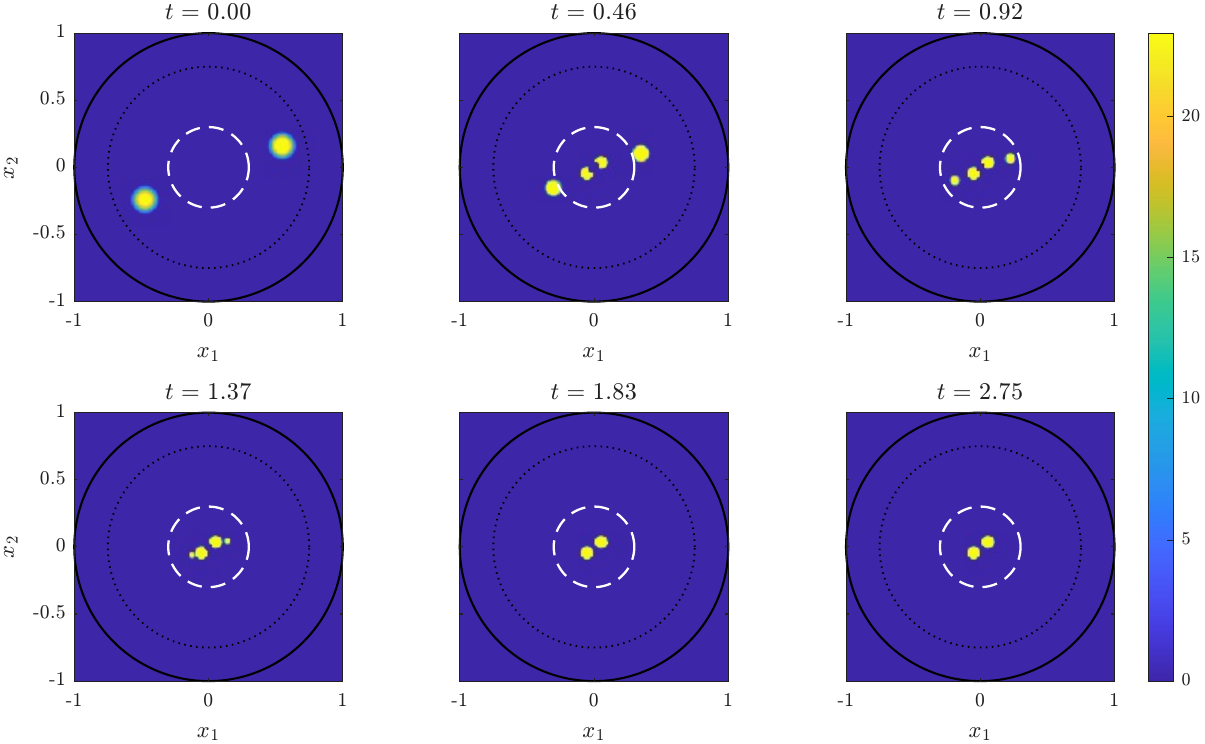}
\caption{Snapshots of the density \(\rho(t,x)\). The density is transported towards \(\rho^\star\) as the transported error enters the active region \(V_d\) at \(t=T_{\mathrm{in}}\) and is damped there. The dashed white circle denotes \(V_d\), and the dotted black circle denotes \(K_0\).}
\label{fig:error_snap}
\end{figure*}

To quantify both the homogeneous decay and the ISS estimate, we consider additive perturbations of the form \(w(t,x)=A_w\sin(\omega_w t)\psi_w(x),\) where \(\psi_w\) is as in \eqref{eq:bump_functions} (supported in a small ball contained in \(K_0\) and close to \(V_d\)), and is normalized in \(L^2(\Omega)\), with \(x_w=(0.38,-0.10)\) and \(\varepsilon_w=0.08\). We use \(\omega_w=8\) and \(A_w\in\{0.5,1,2\}\). Figure \ref{fig:energy_iss} shows the homogeneous energy decay of the error and the input-to-state responses (setting \(e_0=0\)) with amplitudes increasing with \(A_w\), consistently with Theorem \ref{thm:weighted_ISS}. Note that the support of perturbation \(w\) lies in \(K_0\) and for the ISS simulation we chose the fixed support family \(K(t)\equiv K_0\) (\(K_0\) is forward invariant) so that the support hypothesis in Theorem \ref{thm:weighted_ISS} is satisfied.
\begin{figure*}[t]
\centering
\includegraphics[width=0.9\linewidth]{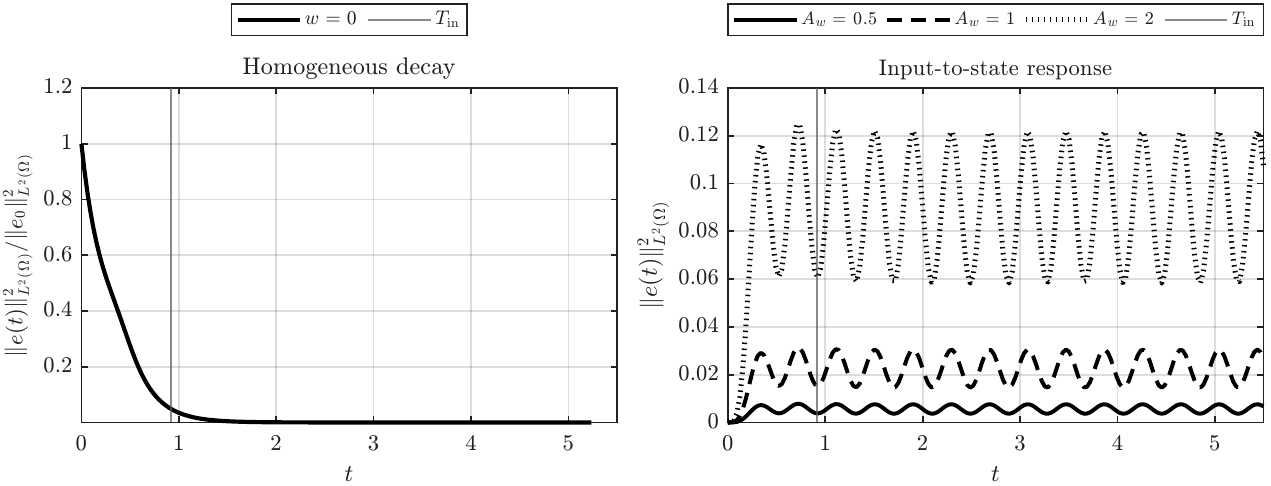}
\caption{Closed-loop error-energy decay and ISS response. Left: homogeneous decay from the initial error \( e_0=\rho_0-\rho^\star\). Right: input-to-state response from zero initial error under bounded additive perturbations of increasing amplitude. The vertical line denotes the entrance time \(T_{\rm in}\).}
\label{fig:energy_iss}
\end{figure*}

All plots here are obtained by evaluating the characteristic representation on a uniform Cartesian grid restricted to \(\Omega\). The inverse characteristic map is used to compute \(e(t,\cdot)\), with numerical integration in time for the accumulated damping along characteristics. The energies are approximated by the corresponding discrete grid sums and for the ISS plot, we use a midpoint rule along backward characteristics.

\subsection{Flux Realization}
We now illustrate how the feedback divergence can be realized by a flux on the actuator region \(U_d\).

Fig. \ref{fig:feedback_divergence} shows the feedback divergence
\(F(t,x)=\kappa\sigma(x)e(t,x)-r^\star(x)\)
at \(t=T_{\rm in}\), when the transported positive component of the initial error has reached the active region \(V_d\), so the localized damping term \(\kappa\sigma e\) is active on the relevant characteristics. We observe that the strongest signed values occur near \(\Sigma_d\), where the feedforward residual \(r^\star\) compensates the transport of the narrow reference profile \(\rho^\star\).

\begin{figure}[t]
\centering
\includegraphics[width=0.78\linewidth]{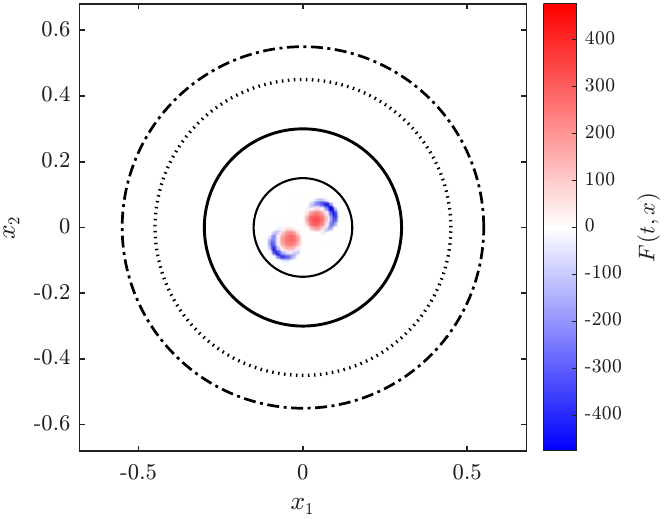}
\caption{Feedback divergence \(F(t,x)=\nabla\cdot J_d(t,x)\) (see \eqref{eq:div_law}) at \(t=T_{\rm in}\). Red and blue denote positive and negative values of \(F\), respectively. The inner solid circle denotes \(\Sigma_d\), the outer solid circle denotes \(V_d\), the dotted circle denotes \(\operatorname{supp}\sigma\), and the dash-dotted circle denotes \(U_d\). The divergence to be constructed by the feedback flux is localized inside \(U_d\).}
\label{fig:feedback_divergence}
\end{figure}

The feedback flux can now be explicitly constructed in \(U_d\) based on the analysis of Section \ref{sec:flux}. Define
\(M_F(t):=\int_{U_d}F(t,y) dy,
\bar F(t):=\frac{1}{|U_d|}M_F(t),
\)
where \(\bar F\) can be written as \(\bar F(t)=\frac{1}{\pi r_U^2}\int_{U_d}F(t,y) dy\) (recalling \(F:=\kappa\sigma e -r^\star\)).
We choose in \eqref{eq:local_neumann_potential} the constant interface flux
\(g_d(t,x)=\frac{M_F(t)}{2\pi r_U}, x\in\partial U_d,\)
which satisfies the compatibility condition
\(\int_{\partial U_d}g_d(t,x)\,dS(x)=M_F(t).\) We first solve the problem
\begin{align*}
-\Delta\eta_0(t,\cdot)
&=F(t,\cdot)-\bar F(t)
\qquad\text{in }U_d, \\
-\nabla\eta_0(t,\cdot)\cdot n_{U_d}&=0 \qquad\text{on }\partial U_d, 
\end{align*}
with the normalization
\(\int_{U_d}\eta_0(t,x) dx=0.\)
Since \(\int_{U_d}\bigl(F(t,x)-\bar F(t)\bigr) dx=0,\)
the elliptic problem satisfies compatibility condition as in Section \ref{sec:flux}. Now, define
\(\eta_b(t,x):=-\frac{M_F(t)}{4\pi r_U^2}|x|^2,\) satisfying \(-\Delta\eta_b(t,x)=\bar{F}(t),\) \( -\nabla\eta_b(t,x)\cdot n_{U_d}= \frac{M_F(t)}{2\pi r_U} \text{on }\partial U_d.\) Therefore, setting \(\eta(t,x):=\eta_0(t,x)+\eta_b(t,x), J_d(t,x):=-\nabla\eta(t,x),\) gives \(\nabla\cdot J_d(t,\cdot)=F(t,\cdot) \text{ in }U_d,\) and \(J_d(t,\cdot)\cdot n_{U_d} = \frac{M_F(t)}{2\pi r_U} \text{ on }\partial U_d.\) We readily obtain \[J_d(t,x) =-\nabla\eta_0(t,x)+ \frac{M_F(t)}{2\pi r_U^2}x, \qquad x\in U_d.\] To represent \(J_d\) more explicitly, let \(G\) denote the Green's function for the Neumann problem for \(-\Delta\) on \(U_d\), namely, \( -\Delta_y G(x,y)=\delta_x(y)-\frac{1}{\vert U_d\vert}\), \(\partial_{n_y}G (x,y)=0\) and zero mean, i.e., \(\int_{U_d}G (x,y) dy=0.\)  Here, \(G\) can be calculated using logarithmic potentials and smooth correction terms. Then, \(\eta_0(t,x)= \int_{U_d}G(x,y) \bigl(F(t,y)-\bar F(t)\bigr) dy,\) and, therefore, inside actuator region \(U_d\), the flux is given in the explicit form
\[J_d(t,x)=-\nabla_x\int_{U_d}G(x,y)\bigl(F(t,y)-\bar F(t)\bigr) dy+ \frac{M_F(t)}{2\pi r_U^2}x.\]
Thus, \(\nabla\cdot J_d=F\) in \(U_d\), and the final radial term accounts for the constant interface flux through \(\partial U_d\).

\subsection{Density-Dependent Velocity Perturbation}

We finally illustrate the nonlinear condition of Section \ref{sec:nonlinear}. Let \(u_0(x):=-\mu q(|x|)x\), as above, reducing to \(u_0(x)=-\mu x\)  on the relevant support family \(K_0\), where \(q\equiv1\). We consider a rotational perturbation of the form
\[u[\rho](t,x)=u_0(x)+\varepsilon b[\rho](t,x), \qquad b[\rho](t,x)=\omega[\rho](t)Rx,\]
\[R=
\begin{pmatrix}
0&-1\\
1&0
\end{pmatrix}, \qquad \omega[\rho](t)=\omega_0\left(1+\frac{\int_\Omega \rho^2(t,y) dy}{1+\int_\Omega \rho^2(t,y) dy}\right),\] 
and take \(\omega_0=3.\) Along any smooth admissible trajectory with \(\rho(t,\cdot)\in L^2(\Omega)\), \(\omega[\rho](t)\) is a bounded nonlinear functional and the velocity field is Lipschitz in space. Moreover, since \(\nabla\Phi(x)=2x\) and \(Rx\cdot x=0\), the perturbation is tangent to the level sets of \(\Phi(x)=|x|^2\), meaning \(b[\rho](t,x)\cdot\nabla\Phi(x)= 2\omega[\rho](t)Rx\cdot x=0.\) It is also tangent to the boundary and the radial part vanishes at \(\partial\Omega\) because \(q(1)=0\).
Therefore, on \(\mathcal K(t)\setminus V_d\), \(D_{u[\rho]}\Phi(x)=D_{u_0}\Phi(x)=-2\mu |x|^2\le-2\mu r_d^2,\)
and Assumption \ref{ass:nonlinear_entr} holds with
\(\beta(r)\equiv 2\mu r_d^2.\)
If \(X_\rho(s;t,x)\) denotes the nonlinear characteristic flow generated by \(u[\rho]\), then for \(x\in K_0\), \(|X_\rho(s;0,x)|=e^{-\mu s}|x|.\) Although the nonlocal functional \(\omega[\rho](t)\) changes the angular motion of the characteristics through the cumulative angle \(\alpha_\varepsilon(t)=\int_0^t\varepsilon\omega[\rho](s) ds\), it does not affect the evolution of their distance from the origin \(|X_\rho(s;0,x)|\). As a result, the entrance time into \(V_d\) is unchanged and this is illustrated in Fig. \ref{fig:nonlinear_rotation}.
\begin{figure*}[!t]
\centering
\includegraphics[width=0.92\textwidth]{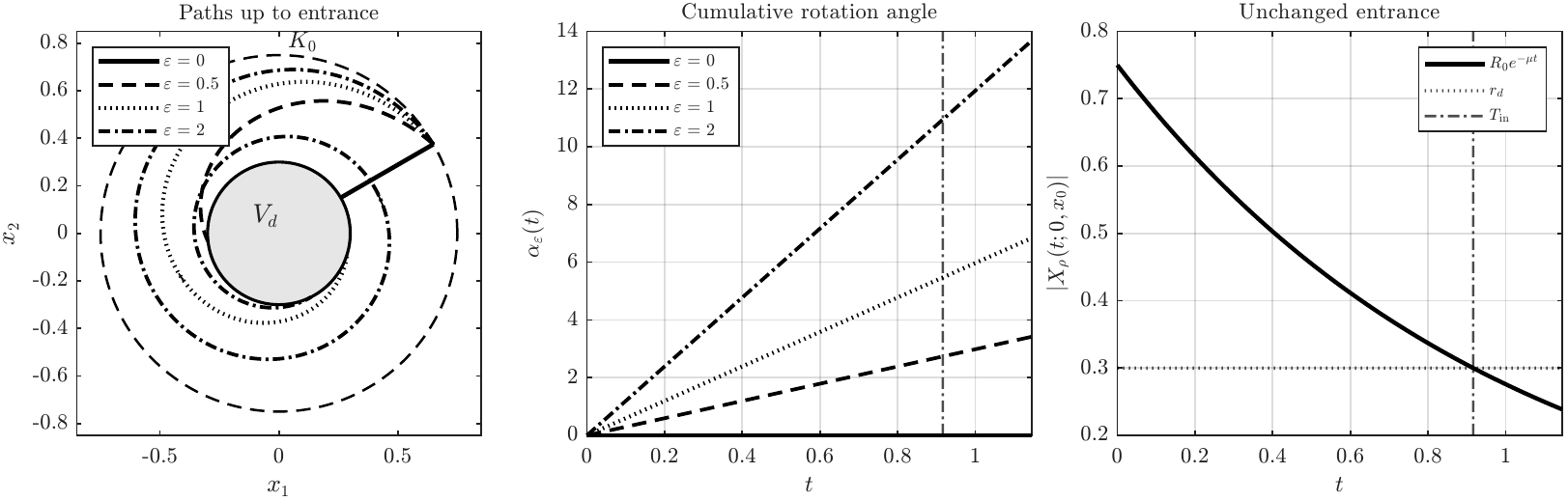}
\caption{Density-dependent rotational perturbation. The first graph shows characteristic paths up to entrance into \(V_d\), the second shows the quantity \(\alpha_\varepsilon(t)=\varepsilon\int_0^t\omega[\rho](s)ds\), through which \(\omega[\rho]\) affects the angular motion, and the third shows that $|X_\rho(t;0,x)|=\exp(-\mu t)|x|$, independently of \(\varepsilon\). We see that the rotational perturbation changes the angular motion but not the entrance time into \(V_d\).} \label{fig:nonlinear_rotation}
\end{figure*}

\begin{remark}
The simple radial geometry above is used only to make the calculations explicit. The framework can also capture anisotropic inward fields, for instance,
\(u(x):=-\mu q(|x|)\bigl(1+\varepsilon_0(x_1^2-x_2^2)\bigr)x; \mu>0, 0\le\varepsilon_0<1,\)
where \(q\in C^\infty([0,1];[0,1])\) is the same type of cutoff function as in the main example, with \(q=1\) on the relevant support region and \(q=0\) near \(\partial\Omega\), with \(u(x)=-\mu\bigl(1+\varepsilon_0(x_1^2-x_2^2)\bigr)x\) on the support family.
With \(\Phi(x)=|x|^2\) again, on the region where \(q=1\), the transport derivative gives \(D_u\Phi(x)=-2\mu |x|^2\bigl(1+\varepsilon_0(x_1^2-x_2^2)\bigr).\)
Since \(|x_1^2-x_2^2|\le |x|^2\le R_0^2<1\) on the support family, it follows that \(D_u\Phi(x)\le -2\mu(1 -\varepsilon_0 R_0^2)a_d, x\in\mathcal K(t)\setminus V_d \) and Assumption \ref{ass:support_lyapunov} holds with
\(\gamma=2\mu(1-\varepsilon_0 R_0^2)a_d.\)
Moreover, on the support family,
\(\nabla\cdot u=-2\mu\bigl(1+2\varepsilon_0(x_1^2-x_2^2)\bigr).\) Consequently,
\(M_{\mathcal K}^-=2\mu(1+2\varepsilon_0 R_0^2),\)
attained along the direction where \(x_1^2-x_2^2\) is maximal on the support family. Corollary \ref{cor:stab_from_entrance} then gives the sufficient asymptotic gain condition \(\kappa>\mu(1+2\varepsilon_0 R_0^2).\) Thus,  this anisotropic field affects the constants entering the developed theory. The entrance estimate depends on the weakest speed through \(1-\varepsilon_0 R_0^2\), while the gain threshold depends on the strongest compressive rate. For the isotropic case \(\varepsilon_0=0\), we recover the threshold \(\kappa>\mu\) of the main example.
\end{remark}

\section{Conclusion}

We developed a localized feedback stabilization framework for continuity equations with interior control fluxes. The main stability structure was finite-time damping accumulated along the relevant characteristics of the velocity field. A support-restricted formulation accounted for compactly supported errors, while a Lyapunov-type entrance condition provided a constructive way to verify the required damping. We also introduced a weighted Lyapunov functional, with a weight defined through a Lyapunov-type function whose sublevel sets determine the active region, and derived ISS estimates for additive perturbations. We presented elliptic right-inverse constructions for producing feedback flux, and finally, we extended the present approach to density-dependent velocity fields. 

Several relevant questions remain open. One direction is the optimal design of the active region \( V_d\) or localization function \(\sigma\), subject to constraints on actuation cost. It would also be interesting to characterize the admissible active regions and Lyapunov-type entrance functions in nonconvex domains or domains with internal obstacles, where the synthesis of velocity fields that guide the characteristics into the active region becomes a geometric control problem. A further extension is the incorporation of positivity, mass, and flux constraints in the realization of the interior control flux, for instance, through optimization-based divergence solvers or coupled continuity equations ensuring mass preservation, as explained in Remark \ref{rem:mass_conservation}.

\section*{References}
\bibliographystyle{IEEEtran}
\bibliography{references}

\begin{IEEEbiography}{}
Constantinos KITSOS received the diploma degree in electrical and computer engineering, the M.Sc. degree in applied mathematics, both from the National Technical University of Athens, Greece, and the Ph.D. in automatic control from Université Grenoble Alpes (GIPSA-lab), Grenoble, France, in 2020. Since then, he has been affiliated as a Postdoctoral Researcher with several research institutes, including the Australian Centre for Robotics, University of Sydney, NSW, Australia, and the Laboratory of Signals and Systems  (L2S), CentraleSupélec, Gif-sur-Yvette, France. 
His research interests include nonlinear observers and control of PDEs. He has also served as an Associate Editor for the European Control Conference 2023 - 2026.
\end{IEEEbiography}

\begin{IEEEbiography}{}
Ian R. MANCHESTER received the B.E. (Hons 1) and Ph.D. degrees in electrical engineering from the University of New South Wales, Sydney, NSW, Australia, in 2002 and 2006, respectively. He has held research positions with Umeå University, Umeå, Sweden, and Massachusetts Institute of Technology, Cambridge, MA, USA. In
2012, he joined the Faculty with the University of Sydney, Sydney, NSW, where he is currently a Professor of mechatronic engineering, Director of the Australian Centre for Robotics, and Director of the Australian Robotic Inspection and Asset Management Hub. His research interests include optimization and learning methods for nonlinear system analysis, identification, and control, and applications in robotics and biomedical engineering. Dr. Manchester is an Associate Editor for IEEE Control Systems Letters and was an Associate Editor for IEEE Robotics and Automation Letters.
\end{IEEEbiography}

\end{document}